\documentclass[11pt,letterpaper]{amsart}
\usepackage{bm,amsmath,amsthm,amssymb,mathtools,verbatim,amsfonts,tikz-cd,mathrsfs}
\usepackage{enumitem}
\usepackage{float}

\usepackage[hidelinks,colorlinks=true,linkcolor=blue, citecolor=black,linktocpage=true]{hyperref}
\usepackage{graphicx}
\usepackage[alphabetic]{amsrefs}
\usepackage{caption}
\usepackage{morefloats}
\usepackage{subfigure}

\usetikzlibrary{matrix,arrows,decorations.pathmorphing}

\usepackage{chngcntr}
\counterwithin{figure}{section}

\newtheorem{thm}{Theorem}[section]
\newtheorem{prop}[thm]{Proposition}
\newtheorem{conj}[thm]{Conjecture}
\newtheorem{cor}[thm]{Corollary}
\newtheorem{lem}[thm]{Lemma}

\theoremstyle{definition}
\newtheorem{define}[thm]{Definition}

\theoremstyle{remark}
\newtheorem{rem}[thm]{Remark}
\newtheorem{example}[thm]{Example}

\newtheorem{question}[thm]{Question}

\newcommand{\ve}[1]{\boldsymbol{\mathbf{#1}}}

\newcommand{\Z}{\mathbb{Z}}
\newcommand{\N}{\mathbb{N}}
\newcommand{\Q}{\mathbb{Q}}

\renewcommand{\d}{\partial}
\renewcommand{\subset}{\subseteq}

\renewcommand{\tilde}{\widetilde}
\renewcommand{\bar}{\overline}

\newcommand{\iso}{\cong}

\DeclareMathOperator{\Char}{{Char}}

\DeclareMathOperator{\Diff}{{Diff}}

\DeclareMathOperator{\gr}{{gr}}

\DeclareMathOperator{\id}{{id}}

\DeclareMathOperator{\im}{{im}}

\DeclareMathOperator{\MOD}{{Mod}}
\DeclareMathOperator{\Mor}{{Mor}}

\DeclareMathOperator{\Spin}{{Spin}}

\DeclareMathOperator{\Tw}{{Tw}}

\DeclareMathOperator{\Span}{{Span}}
\DeclareMathOperator{\Sym}{{Sym}}

\DeclareMathOperator{\Tors}{{Tors}}

\newcommand{\fg}{\mathit{fg}}

\newcommand{\lk}{\mathrm{lk}}

\newcommand{\bA}{\mathbb{A}}
\newcommand{\bB}{\mathbb{B}}
\newcommand{\bC}{\mathbb{C}}

\newcommand{\bE}{\mathbb{E}}
\newcommand{\bF}{\mathbb{F}}

\newcommand{\bH}{\mathbb{H}}
\newcommand{\bI}{\mathbb{I}}

\newcommand{\bP}{\mathbb{P}}

\newcommand{\bT}{\mathbb{T}}

\newcommand{\bZ}{\mathbb{Z}}

\newcommand{\cA}{\mathcal{A}}

\newcommand{\cC}{\mathcal{C}}
\newcommand{\cD}{\mathcal{D}}

\newcommand{\cF}{\mathcal{F}}

\newcommand{\cH}{\mathcal{H}}

\newcommand{\cK}{\mathcal{K}}
\newcommand{\cL}{\mathcal{L}}

\newcommand{\cT}{\mathcal{T}}

\newcommand{\cX}{\mathcal{X}}
\newcommand{\cY}{\mathcal{Y}}
\newcommand{\cZ}{\mathcal{Z}}

\newcommand{\frA}{\mathfrak{A}}

\newcommand{\fra}{\mathfrak{a}}

\newcommand{\frs}{\mathfrak{s}}

\newcommand{\scA}{\mathscr{A}}

\newcommand{\scC}{\mathscr{C}}
\newcommand{\scD}{\mathscr{D}}

\newcommand{\scH}{\mathscr{H}}

\newcommand{\scU}{\mathscr{U}}
\newcommand{\scV}{\mathscr{V}}

\newcommand{\cCFL}{\mathcal{C\!F\!L}}
\newcommand{\cCFK}{\mathcal{C\hspace{-.5mm}F\hspace{-.3mm}K}}

\newcommand{\CF}{\mathit{CF}}

\newcommand{\HF}{\mathit{HF}}

\newcommand{\xs}{\ve{x}}
\newcommand{\ys}{\ve{y}}
\newcommand{\zs}{\ve{z}}
\newcommand{\ws}{\ve{w}}

\newcommand{\as}{\ve{\alpha}}
\newcommand{\bs}{\ve{\beta}}

\renewcommand{\a}{\alpha}
\renewcommand{\b}{\beta}
\newcommand{\g}{\gamma}

\newcommand{\veps}{\varepsilon}

\usepackage{leftidx}

\newcommand{\bCF}{\mathbb{CF}}

\DeclareMathOperator{\Cone}{{Cone}}

\numberwithin{equation}{section}

\newcommand{\ar}{\mathrm{a.r.}}

\newcommand{\llsquare}{[\hspace{-.5mm}[}
\newcommand{\rrsquare}{]\hspace{-.5mm}]}

\newcommand{\vecotimes}{\mathrel{\vec{\otimes}}}

\allowdisplaybreaks

\newcommand{\hatbox}{\mathrel{\hat{\boxtimes}}}

\title{The equivalence of lattice and Heegaard Floer homology}

\author{Ian Zemke}
\address{Department of Mathematics\\Princeton University\\  Princeton, NJ, USA}
\email{izemke@math.princeton.edu}

\begin{document}
\maketitle

\begin{abstract}
 We prove N\'{e}methi's conjecture: if $Y$ is a 3-manifold which is the boundary of a plumbing of a tree of disk bundles over $S^2$, then the lattice homology of $Y$ coincides with the Heegaard Floer homology of $Y$. We also give a conjectural description of the $\Lambda^* H_1(Y)/\Tors$ action when $b_1(Y)>0$.
\end{abstract}

\section{Introduction}

Heegaard Floer homology is a powerful invariant of 3-manifolds, introduced by Ozsv\'{a}th and Szab\'{o} \cite{OSDisks} \cite{OSProperties}. To a closed 3-manifold $Y$, equipped with a $\Spin^c$ structure $\frs\in \Spin^c(Y)$, Ozsv\'{a}th and Szab\'{o} constructed an $\bF[U]$-module denoted $\HF^-(Y,\frs)$. Here and throughout the paper, we work over $\bF=\Z/2$. One defines
\[
\HF^-(Y)=\bigoplus_{\frs \in \Spin^c(Y)} \HF^-(Y,\frs).
\]

An important class of 3-manifolds are the \emph{plumbed} 3-manifolds.  These are the boundaries of 4-manifolds obtained by plumbing a collection of disk bundles over surfaces together. In this paper, we focus on plumbings of trees of disk bundles over $S^2$. Such a plumbing is encoded by an integrally weighted tree $G$.  Given such a $G$, we write $Y(G)$ for the corresponding plumbed 3-manifold.

 Ozsv\'{a}th and Szab\'{o} \cite{OSPlumbed} computed the Heegaard Floer homology of 3-manifolds obtained by plumbing along a tree $G$ when the corresponding 4-manifold is negative definite and the tree has at most one \emph{bad vertex} (i.e. a vertex such that the weight exceeds minus the valence). This family includes all Seifert fibered spaces. 

The computation of Ozsv\'{a}th and Szab\'{o} was formalized by N\'{e}methi \cite{NemethiAR} \cite{NemethiLattice}, who defined an $\bF[U]$-module $\bH\bF(G)$ called \emph{lattice homology}. Lattice homology is the homology of a combinatorially defined chain complex $\bC \bF(G)$. N\'{e}methi \cite{NemethiLattice} conjectured that lattice homology and Heegaard Floer homology are isomorphic for all negative definite plumbing trees $G$.

Work of N\'{e}methi, Ozsv\'{a}th, Stipsicz, and  Szab\'{o} verifies this conjecture in a number of special cases  \cite{OSPlumbed} \cite{OSSLattice} \cite{NemethiAR} \cite{NemethiLattice}  \cite{NemethiRational}, though the general conjecture has remained heretofore open. Additionally, Ozsv\'{a}th, Stipsicz and Szab\'{o} \cite{OSSLattice} construct a spectral sequence from lattice homology to Heegaard Floer homology.



In this paper, we prove the conjecture in full generality. To state our result, we write $\ve{\HF}^-(Y(G))$ for the $\bF\llsquare U\rrsquare$ module obtained by completing $\HF^-(Y(G))$ with respect to the $U$ action. Our main result is the following:

\begin{thm}\label{thm:main}
If $G$ is a plumbing tree, then there is an isomorphism of $\bF\llsquare U\rrsquare $-modules
 \[
 \bH\bF(G)\iso \ve{\HF}^-(Y(G)).
 \]
 When $b_1(Y(G))=0$, the isomorphism is relatively graded.
\end{thm}

 If $Y(G)$ is a rational homology 3-sphere, no information is lost by taking completions with respect to $U$. The same holds if we restrict to torsion $\Spin^c$ structures on a general $Y(G)$. Furthermore, for all $Y$, the reduced module $\HF_{\mathrm{red}}^-(Y)$ is unchanged by completing.
 
 Our proof of Theorem~\ref{thm:main} uses the link surgery formula of Manlescu and Ozsv\'{a}th \cite{MOIntegerSurgery}. There is a well-known presentation of $Y(G)$ as Dehn surgery on a link $L_G$, and we give a direct computation of the link surgery formula of $L_G$ using the techniques of the author \cite{ZemBordered}.


\subsection{Plumbed manifolds with $b_1(Y)>0$}


When $b_1(Y)>0$, Heegaard Floer homology  has an action of $\Lambda^*( H_1(Y)/\Tors)$. Our argument in the case that $b_1(Y(G))>0$ is slightly less natural than the argument when $b_1(Y(G))=0$. For example, we are not able to prove that our isomorphism interwines this homology action when $b_1(Y)>0$.

 In this paper, we state a refinement of N\'{e}methi's conjecture when $b_1(Y(G))>0$. If $\g\in H_1(Y(G))/\Tors$, we define an endomorphism $\frA_{\g}$ on the lattice complex, and prove that our formula gives a well-defined action of $\Lambda^* H_1(Y(G))/\Tors$.  See Section~\ref{sec:plumbed-H1-action}. We give a parallel construction on the link surgery formula in Section~\ref{sec:algebraic-action-link-surgery}.
 We make the following conjecture:

\begin{conj}
\label{conj:intro} $\ve{\HF}^-(Y(G))$ is isomorphic to $\mathbb{HF}(G)$ as a module over $\bF\llsquare U\rrsquare \otimes \Lambda^* H_1(Y(G))/\Tors$.
\end{conj}

In subsequent work \cite{ZemGenSurgery}*{Section~11}, we give an explicit description of the $\Lambda^* H_1(Y(G))/\Tors$ action in terms of the link surgery formula for $L_G$. Since the techniques of \cite{ZemBordered} give a combinatorial model of the surgery formula for $L_G$ which is compatible with this description of the homology action, this reduces Conjecture~\ref{conj:intro} to a purely algebraic question.  Nonetheless, the algebraic arguments of this paper seem insufficient to show that it coincides with the action we describe on lattice homology. 

\subsection{History and overview of the proof}

We now sketch the main ideas that go into the proof, and compare our approach to past approaches. 

The first article directed at computing the Heegaard Floer homology of plumbings is due to Ozsv\'{a}th and Szab\'o \cite{OSPlumbed}. They focused on negative definite plumbings which have at most one \emph{bad vertex} (a bad vertex is a vertex where the weight exceeds minus the valence). 

A key observation of Ozsv\'{a}th and Szab\'{o} was that if $G$ is negative definite, then the 2-handle cobordism $X(G)$ from $S^3$ to $Y(G)$ induces a non-zero map
\[
F_{X(G),\frs} \colon \HF^-(S^3)\to \HF^-(Y(G))
\]
for all $\frs\in \Spin^c(Y(G))$. We may identify $\Spin^c(X(G))$ with the set of characteristic vectors $\Char(X(G))\subset H^2(X(G))$, via $\frs\mapsto c_1(\frs)$. Therefore, for each characteristic vector $K$, we obtain an element
\[
[K]\in \HF^-(Y(G))
\]
 defined by $[K]=F_{X(G),\frs}(1)$ for $c_1(\frs)=K$.

Ozsv\'{a}th and Szab\'{o} \cite{OSPlumbed} define a module
\[
\bH^-(G):=\frac{\bF[U]\otimes \Char(X(G))}{N},
\] where $N$ is an $\bF[U]$-submodule, as follows. For each vertex $v$ in $G$, there is an induced element in $H_2(X_G)$, which is represented by a sphere.  Let
\[
2n=K(v)+v^2.
\]
We let $N$ be the submodule spanned by the elements
\[
\begin{cases}
U^n[K+2v]-[K] & \text{ if } n\ge 0\\
[K+2v]- U^{-n} [K]& \text{ if } n<0,
\end{cases}
\]
for all $K$ and $v$.

 There is a natural map 
 \[
 \Gamma\colon \bH^-(G)\to \HF^-(Y(G))
 \]
 which sends $[K]$ to $F_{X(G), \frs}(1)$ where $c_1(\frs)=K$. Ozsv\'{a}th and Szab\'{o} show this map to be an isomorphism when all weights on $G$ are sufficiently negative. They use the surgery exact triangle and five-lemma to show that the map remains an isomorphism for graphs with at most one bad vertex.

N\'{e}methi \cite{NemethiAR} \cite{NemethiAR} extended this approach, and constructed a chain complex $\bCF(G)$, called the lattice complex. N\'{e}methi focused on a cohomological version of lattice homology, and his definition was most naturally stated for negative definite manifolds. Ozsv\'{a}th, Stipsicz and Szab\'{o} \cite{OSSLattice} \cite{OSSKnotLatticeHomology} gave a homological definition of lattice complex which is defined for arbitrary plumbing trees (e.g. indefinite plumbings). In this paper, we focus on the description given by Ozsv\'{a}th, Stipsicz and Szab\'{o} since it is slightly more convenient for our argument. 

When $Y(G)$ is a rational homology sphere, $\bCF(G)$ has a homological (or Maslov) grading parallel to the Heegaard Floer homology groups. There is an additional grading, called the \emph{lattice grading} or \emph{cube grading}, which the differential decreases by 1.  The homology of $\bCF(G)$ in lattice grading 0 coincides with the module $\bH^-(G)$ defined earlier by Ozsv\'{a}th and Szab\'o.

 N\'{e}methi gave many conceptual insights into the lattice complex, especially in terms of its connection to complex surface singularities and algebraic geometry. N\'{e}methi \cite{NemethiRational} showed that if $G$ is negative definite, then $G$ is \emph{rational} (cf.  \cite{Artin-Rational-1} \cite{Artin-Rational-2})  if and only if $Y(G)$ is an $L$-space. N\'{e}methi \cite{NemethiAR} also showed that Ozsv\'{a}th and Szab\'{o}'s one-bad-vertex condition could be generalized to a condition called \emph{almost rationality}, which has a more natural interpretation in terms of algebraic geometry. A negative definite graph $G$ is defined to be \emph{almost rational} if it becomes rational after decreasing the weight at one vertex. 

One challenge of generalizing the isomorphism of Ozsv\'{a}th, Szab\'{o} and N\'{e}methi in the almost rational case is that it is rather challenging to define a natural map from $\mathbb{HF}(G)$ to $\ve{\HF}^-(Y(G))$ for general $G$. One natural approach to constructing such an isomorphism would be to define a canonical chain map 
\[
\Gamma\colon \mathbb{CF}(G)\to \ve{\CF}^-(Y(G))
\]
 which commutes with the maps in certain surgery exact triangles. Assuming $\Gamma$ to be an isomorphism for graphs with sufficiently negative framings, one could hope to use the five-lemma to prove the isomorphism for all $G$. It turns out to be quite challenging to define a map $\Gamma$ as above for general plumbing trees $G$. If such a map existed, one would expect that the component of $\Gamma$ from $\mathbb{CF}_q(G)$ to $\ve{\CF}^-(Y(G))$ to count holomorphic $(3+q)$-gons. Here $q$ denotes the lattice grading. Such a construction has not been achieved directly.

Our approach uses the link surgery formula of Manolescu and Ozsv\'{a}th \cite{MOIntegerSurgery}. To a link $L\subset S^3$ with integral framing $\Lambda$, Manolescu and Ozsv\'{a}th define a chain complex $\cC_{\Lambda}(L_G)$ whose homology is $\ve{\HF}^-(Y(G))$. 

The manifold $Y(G)$ has a convenient surgery description as integral surgery on a link $L_G\subset S^3$. The link $L_G$ has one unknotted component for each vertex of $G$, and a clasp for each edge of $G$. The weights give an integral framing $\Lambda$ on $L_G$, and $Y(G)\iso S^3_{\Lambda}(L_G)$.

A key observation of Ozsv\'{a}th, Stipsicz and Szab\'{o} \cite{OSSLattice}*{Proposition~4.4} is that the lattice complex $\mathbb{CF}(G)$ is closely related to the link surgery complex $\cC_{\Lambda}(L_G)$. By construction, $\cC_{\Lambda}(L_G)$ is a hypercube of chain complexes. This means that the underlying module of $\cC_{\Lambda}(L_G)$ is a direct sum of subspaces $\cC_{\veps}$, indexed by points $\veps\in \{0,1\}^n$, and the differential on $\cC_{\Lambda}(L_G)$ decomposes as a sum $D=\sum_{\veps\le \veps'} D_{\veps,\veps'}$ where $D_{\veps,\veps'}$ maps $\cC_{\veps}$ to $\cC_{\veps'}$. We refer to the quantity $n-|\veps|_{L^1}$ as the \emph{cube filtration}. We may write differential $D$ on $\cC_{\Lambda}(L_G)$ as a sum $D=D_0+D_1+\cdots+D_n$ where $D_i$ decreases cube filtration by $i$. Ozsv\'{a}th, Stipsicz and Szab\'{o} show that there is an isomorphism of chain complexes
 \[
\mathbb{CF}(G)\iso  \left(\bigoplus_{\veps\in \{0,1\}^n} H_*(\cC_{\veps},D_0),(D_1)_*\right),
 \]
  where $(D_1)_*$ is the map induced by $D_1$ on after taking homology with respect to $D_0$. In particular, the $E_2$ page of the spectral sequence from the cube filtration on $\cC_{\Lambda}(L_G)$ is lattice homology, so one implication of Ozsv\'{a}th, Stipsicz and Szab\'{o}'s work is the existence of a spectral sequence from lattice homology to Heegaard Floer homology.
 
 Our strategy is to carefully study the link surgery complex $\cC_{\Lambda}(L_G)$, and to construct directly an isomorphism between $\cC_{\Lambda}(L_G)$ and $\mathbb{CF}(G)$. We use a number of technical results proven in \cite{ZemBordered}. Two key results proven therein are the following:
 \begin{enumerate}
 \item A connected sum formula for Manolescu and Ozsv\'{a}th's link surgery formula.
 \item A computation of the link surgery complex of the Hopf link.
 \end{enumerate} 
  Since the link $L_G$ is an iterated connected sum of Hopf links, the above two results yield a combinatorial model for $\cC_{\Lambda}(L_G)$. In the present paper, we use homological perturbation theory to show that $\cC_{\Lambda}(L_G)$ is homotopy equivalent to a complex obtained by adding extra (higher length) differentials to the lattice complex. We then show that these higher length differentials may be homotoped away, so that the conjectured isomorphism is obtained. 

There are two main challenges we face in our reduction of higher length arrows:
\begin{enumerate}
\item \emph{Selection of a suitable arc system for the Hopf link complex}. It was observed in \cite{ZemBordered} that the link surgery formula depends on a choice of auxiliary data, called an \emph{arc system}. This is parallel to the traditional bordered Floer setting \cite{LOTBordered}, where a bordered Heegaard diagram has some closed alpha and beta circles, as well as either alpha or beta \emph{arcs} which are asymptotic to the boundary of the Heegaard surface. Using alpha arcs instead of beta arcs (or vice versa) can change the resulting bordered module. The analogous phenomenon in the setting of the link surgery complex is analyzed in \cite{ZemBordered}*{Sections~14, 16}. One choice of arc system for the Hopf link adds comparably fewer higher length differentials to $\cC_{\Lambda}(L_G)$, and we use this one in our proof.
\item \emph{Simplification of the connected sum formula}. Similar to connected sum formulas in other Heegaard Floer contexts (see e.g. \cite{ZemConnectedSums}), our connected sum formula for the link surgery formula has some extra terms involving analogs of the basepoint actions which appear on knot Floer homology. In general, these terms will add higher length arrows to the link surgery complex of connected sums. In this paper, we describe several techniques to eliminate these terms in certain settings. See Corollary~\ref{cor:simplify-tensor-product} and Proposition~\ref{prop:merge=pair-of-pants}. 
\end{enumerate}

Our perspective and strategy is largely influenced by the bordered Floer homology theory of Lipshitz, Ozsv\'{a}th and Thurston \cite{LOTBordered}. Indeed, in \cite{ZemBordered} we show that the link surgery formula of Manolescu and Ozsv\'{a}th is naturally reinterpreted as a theory for 3-manifolds with parametrized torus boundary components. There seems to be a close connection between the bordered theory from \cite{ZemBordered} and the bordered theory of Lipshitz, Ozsv\'{a}th and Thurston. Many of the arguments we use in this paper were inspired by parallel developments in the traditional bordered setting, which we translated into the setting of the link surgery formula. 

\subsection{Further questions}

In this section, we list some important remaining questions about lattice homology.

\begin{question}Is there a natural extension of lattice homology which computes the involutive Floer theory of Hendricks and Manolescu \cite{HMInvolutive}?
\end{question}
\begin{rem} In the almost rational case, the involution turns out to be the natural $J$ map \cite{dai-manolescu}. See \cite{NemethiLattice}*{Remark 3.2.7} \cite{OSSKnotLatticeHomology}*{Section 2.2}  for the definition of $J$. The expectation of the community seems to be that $J$ is unlikely to be the correct involution in general. One justification is that $J^2=\id$ on the chain level. This relation holds even for the associated versions of knot lattice homology \cite{OSSKnotLatticeHomology}, where the knot involution $\iota_K$ of Hendricks and Manolescu \cite{HMInvolutive}*{Section~6.1} does not in general square to the identity on the chain level, even up to chain homotopy. A more reasonable conjecture might be that $J$ coincides with the involution on the level of homology, but not always on the chain level.
\end{rem}

\begin{question}
Is there a more intrinsic or topological interpretation of the lattice grading?
\end{question}
\begin{rem}
The existence of the extra lattice grading has some non-trivial applications. For example in \cite{BLZLatticeLink}, we prove the equivalence of link lattice homology with link Floer homology for plumbed links in rational homology 3-spheres. The extra lattice grading allows us to prove that the link Floer complexes of plumbed $L$-space links (a family including all algebraic links) are formal.
\end{rem}

\begin{question}
Is there a natural extension of lattice homology which computes $\ve{\HF}^-(Y(G))$ for more general plumbings of disk bundles over surfaces, including self-plumbings and higher genus surfaces?
\end{question}

\subsection{Acknowledgments}

The author would like to thank Antonio Alfieri, Maciej Borodzik, Kristen Hendricks, Jennifer Hom, Robert Lipshitz, Beibei Liu, Ciprian Manolescu, Peter Ozsv\'{a}th and Matt Stoffregen  for helpful conversations. We also thank the anonymous referees for their careful reading and thoughtful comments.

\section{Background}

\subsection{Heegaard Floer homology}
\label{sec:background}

In this section, we recall some background on Heegaard Floer homology and its refinements for knots and links. 

Heegaard Floer homology is an invariant of closed 3-manifolds defined by Ozsv\'{a}th and Szab\'{o} \cite{OSDisks} \cite{OSProperties}. Given a pointed Heegaard diagram $(\Sigma,\as,\bs,w)$ for $Y$, one considers the Lagrangian tori 
 \[
 \bT_{\a}=\a_1\times \cdots \times\a_g\quad \text{and} \quad \bT_{\b}=\b_1\times \cdots \times\b_g,
 \]
 inside of $\Sym^g(\Sigma)$. The chain complex $\CF^-(Y,\frs)$ is freely generated over $\bF[U]$ by intersection points $\xs\in \bT_{\a}\cap \bT_{\b}$
satisfying $\frs_w(\xs)=\frs$. The differential counts index 1 pseudoholomorphic disks $u$ weighted by $U^{n_w(u)}$, where $n_w(u)$ is the intersection number of the image of $u$ with $\{w\}\times \Sym^{g-1}(\Sigma)$.

We now recall the construction of knot and link Floer homology. Knot Floer homology is due to Ozsv\'{a}th and Szab\'{o} \cite{OSKnots} and independently Rasmussen \cite{RasmussenKnots}. Link Floer homology is due to Ozsv\'{a}th and Szab\'{o} \cite{OSLinks}. We focus on the description in terms of a free chain complex over a 2-variable polynomial ring $\bF[\scU,\scV]$. Given a doubly pointed Heegaard diagram $(\Sigma,\as,\bs,w,z)$ representing a pair $(Y,K)$ consisting of a knot $K$ in a rational homology 3-sphere $Y$, one defines $\cCFK(Y,K)$ to be the free $\bF[\scU,\scV]$-module generated by intersection points $\bT_{\a}\cap \bT_{\b}$. The differential counts Maslov index 1 pseudo-holomorphic disks $u$ which are weighted by $\scU^{n_{w}(u)}\scV^{n_{z}(u)}$. 

For an $\ell$-component link $L\subset Y$, we consider a $2\ell$-pointed Heegaard link diagram $(\Sigma,\as,\bs,\ws,\zs)$. Here,
\[
\ws=\{w_1,\dots, w_\ell\}\quad \text{and} \quad \zs=\{z_1,\dots, z_\ell\}
\] are collections of basepoints on $L$. If we write $K_1,\dots, K_\ell$ for the components of $L$, then we assume that $w_i,z_i\in K_i$. 

 We define $\cCFL(Y,L)$ to be the complex freely generated over the ring $\bF[\scU_1,\scV_1,\dots, \scU_{\ell}, \scV_{\ell}]$ by intersection points $\xs\in \bT_{\a}\cap\bT_{\b}$. This version of link Floer homology was considered in \cite{ZemCFLTQFT}. In the differential, a holomorphic disk $u$ is weighted by the algebra element 
 \[
 \scU_1^{n_{w_1}(u)}\cdots \scU_\ell^{n_{w_\ell}(u)} \scV_1^{n_{z_1}(u)}\cdots \scV_\ell^{n_{z_\ell}(u)}.
 \] 

We recall that when $L$ is a link in a rational homology 3-sphere $Y$, there is an $\ell$-component $\Q^\ell$-valued Alexander grading $A=(A_1,\dots, A_\ell)$ on $\cCFL(Y,L)$. The most important case for our purposes is when $Y=S^3$. In this case the Alexander grading takes values in the set
\[
\bH(L):=\prod_{i=1}^\ell \Z+\frac{\lk(K_i,L-K_i)}{2}.
\]
 The variable $\scU_i$ has $A_j$-grading $-\delta_{i,j}$ (the Kronecker delta) and $\scV_i$ has $A_j$-grading $\delta_{i,j}$.

There are additional \emph{basepoint actions} on link Floer homology which make an appearance in our paper. They appear frequently when studying knot and link Floer homology (see, e.g.  \cite{SarkarMaslov} or \cite{ZemQuasi}). For each $i\in \{1,\dots, \ell\}$, there are endomorphisms $\Phi_{w_i}$ and $\Psi_{z_i}$ of $\cCFL(Y,L)$, defined as follows. We pick a free $\bF[\scU_1,\scV_1,\dots, \scU_\ell,\scV_{\ell}]$-basis of $\cCFL(Y,L)$. We write $\d$ as a matrix with $\bF[\scU_1,\scV_1,\dots, \scU_\ell,\scV_{\ell}]$ coefficients with respect to this basis. The map $\Phi_{w_i}$ is obtained by formally differentiating this matrix with respect to $\scU_i$. The map $\Psi_{z_i}$ is obtained by formally differentiating this matrix with respect to $\scV_i$.

\subsection{Hypercubes and hyperboxes}

In this section we recall Manolescu and Ozsv\'{a}th's notion of a hypercube of chain complexes, as well as versions in the Fukaya category. See \cite{MOIntegerSurgery}*{Section~5 and 8}.
We write $\bE_n=\{0,1\}^n$. If $\ve{d}=(d_1,\dots, d_n)$ is a tuple with $d_i>0$, we write $\bE(\ve{d})=\{0,\dots, d_1\}\times \cdots \times \{0,\dots, d_n\}$.
We write $\veps\le \veps'$ if inequality holds at all coordinates. We write $\veps<\veps'$ if $\veps\le \veps'$ and $\veps\neq \veps'$. Throughout this paper, we work over $\Z/2$.

\begin{define} A \emph{hypercube of chain complexes} consists of a collection of $\Z/2$-modules $C_\veps$, ranging over all $\veps\in \bE_n$, and maps $D_{\veps,\veps'}\colon C_{\veps}\to C_{\veps'}$, ranging over all pairs $\veps,\veps'$ such that $\veps\le \veps'$. We assume furthermore that whenever $\veps\le \veps''$ we have
\begin{equation}
\sum_{\substack{\veps'\in \bE_n \\\veps\le \veps'\le \veps''}} D_{\veps',\veps''}\circ D_{\veps,\veps'}=0.
\label{eq:hypercube-relations}
\end{equation}
\end{define}

A \emph{hyperbox of chain complexes of size $\ve{d}\in (\Z^{>0})^n$} is similar: We assume that we have a collection of groups $C_{\veps}$ ranging over $\veps\in \bE(\ve{d})$, as well as a collection of linear maps $D_{\veps,\veps'}$ ranging over all $\veps\le \veps'$ such that $|\veps'-\veps|_{L^\infty}\le 1$. We assume that Equation~\eqref{eq:hypercube-relations} holds whenever $|\veps''-\veps|_{L^\infty}\le 1$.

For our purposes, it is also important to consider a notion of hypercubes in the Fukaya category (see \cite{MOIntegerSurgery}*{Section~8.2}): 

\begin{define}
\label{def:hypercube-betas}
 A \emph{hypercube of beta-attaching curves} $\cL_{\b}$ on $(\Sigma,\ws)$ consists of a collection of attaching curves $\bs_{\veps}$ ranging over $\veps\in \bE_n$, as well as a collection of chains (i.e. morphisms in the Fukaya category)
\[
\Theta_{\veps,\veps'}\in \ve{\CF}^-(\Sigma,\bs_{\veps},\bs_{\veps'})
\]
whenever $\veps<\veps'$, satisfying the following compatibility condition for each pair $\veps<\veps'$:
\[
\sum_{\veps=\veps_1<\dots<\veps_j=\veps'} f_{\b_{\veps_1},\dots, \b_{\veps_j}}(\Theta_{\veps_1,\veps_2},\dots, \Theta_{\veps_{j-1},\veps_j})=0.
\]
\end{define}
We usually assume that the diagram containing all $2^n$ attaching curves is weakly admissible.
 For the purposes of this paper, it is also sufficient to consider only hypercubes where each pair $\bs_{\veps}$ and $\bs_{\veps'}$ are handleslide-equivalent on a multi-pointed Heegaard surface $(\Sigma,\ws)$.

The Floer complex $\ve{\CF}^-(\Sigma,\bs_{\veps},\bs_{\veps'})$ appearing in Definition~\eqref{def:hypercube-betas} is freely generated over $\bF\llsquare U_1,\dots, U_n\rrsquare$ by intersection points $\xs\in \bT_{\b_{\veps}}\cap \bT_{\b_{\veps'}}$. The map $f_{\b_{\veps_1},\dots, \b_{\veps_j}}$ appearing in Definition~\ref{def:hypercube-betas} counts pseudoholomorphic polygons $u$ of index $3-j$ which are weighted by $U_1^{n_{w_1}(u)}\cdots U_n^{n_{w_n}(u)}$.

A hypercube of \emph{alpha}-attaching curves $\cL_{\a}=(\as_{\veps},\Theta_{\veps',\veps})_{\veps\in \bE_{n}}$ is defined similarly to Definition~\ref{def:hypercube-betas}, except that we have chains $\Theta_{\veps',\veps}\in \ve{\CF}^-(\Sigma,\as_{\veps'},\as_{\veps})$ whenever $\veps<\veps'$.

Given hypercubes of alpha and beta attaching curves $\cL_{\a}$ and $\cL_{\b}$, we may pair them and form a chain complex $\ve{\CF}^-(\cL_{\a},\cL_{\b})$, which is a hypercube of chain complexes of dimension $n+m$, where $n=\dim \cL_\a$ and $m=\dim \cL_{\b}$. The underlying chain complex of the pairing at cube point $(\veps,\nu)\in \bE_n\times \bE_m$ is $\ve{\CF}^-(\as_{\veps}, \bs_{\nu})$. The hypercube differential is defined as follows. We set $D_{(\veps,\nu),(\veps,\nu)}$ to be the ordinary Floer differential. If $(\veps,\nu)< (\veps',\nu')$, then
\begin{equation}
\begin{split}
&D_{(\veps,\nu),(\veps',\nu')}(\xs)\\
=&\sum_{\substack{\veps=\veps_1<\dots<\veps_i=\veps'\\ \nu=\nu_1<\cdots<\nu_j=\nu'}} f_{\a_{\veps_i},\dots, \a_{\veps_1}, \b_{\nu_1},\dots, \b_{\nu_j}}(\Theta_{\a_{\veps_i},\a_{\veps_{i-1}}},\dots,\xs,\dots, \Theta_{\b_{\nu_{j-1}},\b_{\nu_j}}).
\end{split}
\label{eq:hypercube-differential}
\end{equation}

\subsection{Background on lattice homology}

In this section, we recall the definition of lattice homology. Lattice homology is an invariant of plumbed 3-manifolds due to N\'{e}methi \cite{NemethiAR} \cite{NemethiLattice}, which is a formalization of Ozsv\'{a}th and Szab\'{o}'s computation of the Heegaard Floer homology of some plumbed 3-manifolds \cite{OSPlumbed}.  We will use the notation of  Ozsv\'{a}th, Stipsicz and Szab\'{o} \cite{OSSLattice} because of its relation to the surgery formula of Manolescu and Ozsv\'{a}th. 
 
 If $G$ is a plumbing tree, then we write $Y(G)$ for the associated 3-manifold and $X(G)$ for the associated 4-manifold, which has boundary $Y(G)$. We recall that $K\in H^2(X(G))$ is a \emph{characteristic vector} if
 \[
 K (\Sigma)+\Sigma^2\equiv 0\pmod 2,
 \]
 for every class $\Sigma\in H_2(X(G))$. We write $\Char(G)$ for the set of characteristic vectors of $X(G)$.

We now sketch the definition of $\mathbb{CF}(G)$. Write $V(G)$ for the set of vertices, and $\bP(G)$ for the power set of $V(G)$. Generators of the complex are written $[K,E]$ where $K\in \Char(X(G))$ and $E\in \bP(G)$. The lattice complex is defined as
 \[
 \bCF(G)=\prod_{\substack{
 K\in \Char(G)\\
 E\in \bP(V)}} \bF\llsquare U\rrsquare \otimes \langle [K,E]\rangle.
 \]
 The differential on $\bCF(G)$ is defined via the equation
 \begin{equation}
 \d[K,E]=\sum_{v\in E} U^{a_v(K,E)}\otimes [K,E-v]+\sum_{v\in E} U^{b_v(K,E)}\otimes [K+2v^*,E-v], \label{eq:lattice-differential}
 \end{equation}
 extended equivariantly over $U$. In the above equation, $a_v(K,E)$ and $b_v(K,E)$ denote certain nonnegative integers. See \cite{OSSKnotLatticeHomology}*{Section~2} for the definition in our present notation.

\subsection{Knot and link surgery formulas}
\label{sec:background-surgery-formulas}

We now recall some basics of the Manolescu--Ozsv\'{a}th link surgery formula \cite{MOIntegerSurgery}, as well as the knot surgery formulas of Ozsv\'{a}th and Szab\'{o} \cite{OSIntegerSurgeries} \cite{OSRationalSurgeries}. Additional perspectives and background on the surgery formulas may be found in \cite{ZemGenSurgery}. 

If $K$ is a knot in an integer homology 3-sphere $Y$, then Ozsv\'{a}th and Szab\'{o} \cite{OSIntegerSurgeries} proved a formula which relates $\HF^-(Y_n(K))$ with the knot Floer complex of $K$. They defined two chain complexes $\bA(Y,K)$ and $\bB(Y,K)$ over $\bF\llsquare U\rrsquare $, as well as two $\bF\llsquare U\rrsquare $-equivariant maps $v,h_n\colon \bA(Y,K)\to \bB(Y,K)$ such that
\[
\ve{\HF}^-(Y_n(K))\iso H_* \Cone(v+h_n\colon \bA(Y,K)\to \bB(Y,K)).
\]
Here, $\ve{\HF}^-$ denotes $\HF^-$ with coefficients in $\bF\llsquare U\rrsquare$ and the modules $\bA(Y,K)$ and $\bB(Y,K)$ are suitable completions of $\cCFK(Y,K)$ and $\scV^{-1} \cCFK(Y,K)$, respectively.

We now describe the maps $v$ and $h_n$.  The map $v$ is the canonical inclusion map for localization at $\scV$. The map $h_n$ is defined by composing the canonical inclusion map of $\cCFK(Y,K)$ into $\scU^{-1} \cCFK(Y,K)$, with a homotopy equivalence of $\bF[U]$-chain complexes 
\[
\scU^{-1} \cCFK(Y,K)\simeq \scV^{-1}\cCFK(Y,K),
\]
which we define presently.
There are two canonical isomorphisms 
 \[
\begin{split}
 \theta_z&\colon \scU^{-1} \cCFK(Y,K)\to \CF^-(Y,z)\otimes \bF[T,T^{-1}]\\
\theta_w& \colon \scV^{-1} \cCFK(Y,K)\to \CF^-(Y,w)\otimes \bF[T,T^{-1}],
\end{split} 
\] which are $\bF[U]$-equivariant and preserve the Alexander grading. Here $T$ denotes a formal variable. We put $\CF^-(Y,z)$ in Alexander grading $0$ and define $A(T)=1$. There is a diffeomorphism map $\lambda_*\colon \CF^-(Y,z)\to \CF^-(Y,w)$ obtained by pushing $z$ to $w$ along an arc on the Heegaard surface. The map $h_n$ is defined as the composition
 \[
 h_n= \scV^n \cdot \theta_w^{-1} \circ (\lambda_*\otimes \id_{\bF[T,T^{-1}]})\circ \theta_z.
 \] 
 When $Y$ is a rational homology 3-sphere, the choice of arc $\lambda$ does not change the homotopy class of the map $h_n$ by \cite{ZemGraphTQFT}*{Theorem~D}. For general $Y$, the choice of arc is important, and we pick $\lambda$ to be one of the two components of $K\setminus \{w,z\}$. 

Note that the map $h_n$ shifts the Alexander grading by $n$, while $v$ preserves the Alexander grading.

Manolescu and Ozsv\'{a}th extended the knot surgery formula to links in $S^3$ \cite{MOIntegerSurgery}. To a link $L\subset S^3$ with integral framing $\Lambda$, they constructed a chain complex $\cC_{\Lambda}(L)$ whose homology is $\ve{\HF}^-(S^3_{\Lambda}(L))$. The chain complex $\cC_{\Lambda}(L)$ is built from the link Floer complex of $L$. The chain complex $\cC_{\Lambda}(L)$ is an $\ell$-dimensional hypercube of chain complexes, where $\ell=|L|$. 

The construction of the link surgery formula also requires a choice of arcs on the Heegaard surface, which we are presently suppressing from the notation. See Sections~\ref{sec:system-of-arcs} and~\ref{sec:basic-systems} for details on this. Unlike in the surgery formula for knots in homology spheres, the choice of arcs can change the complex, even for links in $S^3$.

We usually conflate the integral framing $\Lambda$ with the symmetric framing matrix, which has $\Lambda_{i,i}$ equal to the framing of $L_i$, and $\Lambda_{i,j}=\lk(L_i,L_j)$ if $i\neq j$.

 The underlying group of $\cC_{\Lambda}(L)$ has a simple description in terms of link Floer homology. If $\veps\in \bE_\ell$, write $S_{\veps}$ for the multiplicatively closed subset of $\bF[\scU_1,\scV_1,\dots, \scU_{\ell}, \scV_\ell]$ generated by $\scV_i$ for $i$ such that $\veps_i=1$. Then $\cC_{\veps}$ is a completion of  $S_\veps^{-1}\cdot \cCFL(L).$
(This is a slight reformulation of Manolescu and Ozsv\'{a}th's original description; see \cite{ZemBordered}*{Lemma~7.5}).

The hypercube differential decomposes over sublinks of $L$ which are oriented (possibly differently than $L$). If $\vec{M}$ is an oriented sublink of $L$, we write $\Phi^{\vec{M}}$ for the corresponding summand of the differential. If $\veps< \veps'$ and $\veps'-\veps$ is the indicator function for the components of $M$, then $\Phi^{\vec{M}}$ sends $\cC_{\veps}$ to $\cC_{\veps'}$. 

For each $\veps$, the group $\cC_{\veps}\iso  S_{\veps}^{-1}\cdot \cCFL(L)$ decomposes over Alexander gradings $\ve{s}\in \bH(L)$. We write $\cC_{\veps}(\ve{s})\subset S_{\veps}^{-1}\cdot \cCFL(L)$ for this subgroup. The group $\cC_{\veps}(\ve{s})$ is preserved by the action of $U_i=\scU_i\scV_i$, for $i\in \{1,\dots, \ell\}$, as well as the internal differential from $\cCFL(L)$, though it is not preserved by the actions of $\scU_i$ or $\scV_i$.

The hypercube maps $\Phi^{\vec{M}}$ have a predictable effect on Alexander gradings. If $\vec{M}\subset L$ is an oriented sublink, define $\Lambda_{L,\vec{M}}\in \Z^\ell$ as follows. We may canonically identify $\Z^\ell$ with $H_1(S^3\setminus \nu(L))$. Under this isomorphism, the element $(0,\dots, 1,\dots 0)\in\Z^\ell$ is identified with a meridian of a component of $L$. We define $\Lambda_{L,\vec{M}}$ to be the sum of the longitudes of link components $K_i$ in $\vec{M}$ such that the orientations from $\vec{M}$ and $L$ are opposite. With respect to this notation, 
\begin{equation}
\Phi^{\vec{M}}(\cC_{\veps}(\ve{s}))\subset \cC_{\veps'}(\ve{s}+\Lambda_{L,\vec{M}})\label{eq:grading-link-surgery}
\end{equation}
where $\veps'-\veps$ is the indicator function for $M$.

Ozsv\'{a}th, Stipsicz and Szab\'{o} describe an important relation between the link surgery and lattice complexes which we use in our proof.  Let $(\cC_{\veps},D_{\veps,\veps'})_{\veps\in \bE_n}$ denote the link surgery hypercube $\cC_{\Lambda}(L_G)$, where $G$ is a plumbing tree. We may construct another hypercube $(H_{\veps}, d_{\veps,\veps'})$ by setting $H_{\veps}=H_* (\cC_{\veps})$, and by setting $d_{\veps,\veps'}=(D_{\veps,\veps'})_*$ if $|\veps'-\veps|_{L^1}=1$ and $d_{\veps,\veps'}=0$ otherwise. Note that $\cC_{\veps}$ is naturally a module over $\bF\llsquare U_1,\dots, U_{\ell}\rrsquare $, where $\ell=|L_G|$, however each $U_i$ has the same action on homology, so we view $H_* (\cC_{\veps})$ as being an $\bF\llsquare U\rrsquare $-module, where $U$ acts by any of the $U_i$.   Ozsv\'{a}th, Stipsicz and Szab\'{o} prove the following:
 
 \begin{prop}[\cite{OSSLattice}*{Proposition~4.4}]
 \label{prop:OSS-lattice}
  There is an isomorphism of hypercubes of chain complexes over $\bF\llsquare U\rrsquare $:
  \[
  (H_\veps,d_{\veps,\veps'})\iso \bC \bF(G)
  \]
  which is relatively graded on torsion $\Spin^c$ structures. In particular, the total homologies of the two sides coincide.
 \end{prop}

\subsection{Systems of arcs}
\label{sec:system-of-arcs}

In this section, we recall the notion of a \emph{system of arcs} for a link $L$, which is a piece of auxiliary data necessary to build the link surgery formula.

\begin{define} Suppose that $L\subset S^3$ is a link such that each component $K_i\subset L$ is equipped with a pair of basepoints, denoted $w_i$ and $z_i$. A \emph{system of arcs} $\scA$ for $L\subset S^3$ consists of a collection of $\ell=|L|$ embedded and pairwise disjoint arcs $\lambda_1,\dots, \lambda_\ell$ arcs, such that
 \[
\lambda_i \cap L=\d \lambda_i \cap L=\{w_i,z_i\}.
\]
\end{define}

If $L$ is oriented, we say an arc $\lambda_i$ is \emph{beta-parallel} if it is a small push-off of the segment of $K_i$ which is oriented from $z_i$ to $w_i$. We say that $\lambda_i$ is \emph{alpha-parallel} if it is a small push-off of the segment of $K_i$ oriented from $w_i$ to $z_i$.

The construction of Manolescu and Ozsv\'{a}th focuses on arc systems where all of the arcs are alpha-parallel. Their proof also applies with little change to the case that each arc is either alpha-parallel or beta-parallel. We write $\cC_{\Lambda}(L,\scA)$ for the link surgery complex computed with the arc system $\scA$.

In \cite{ZemBordered}*{Section~13}, the author studied the effect of changing the arc system, and proved a general formula which computes the effect of changing the arc system. See \cite{ZemBordered}*{Corollary~13.5} for a precise statement. 

%
%
%
%
%

\subsection{Construction of the surgery hypercube}
\label{sec:basic-systems}
In this section, we sketch the construction of the link surgery hypercube. Manolescu and Ozsv\'{a}th's construction requires a large collection of Heegaard diagrams, which they refer to as a \emph{complete system of Heegaard diagrams}. We describe their construction in a restricted setting, focusing on the case that each arc in our system of arcs $\scA$ for $L$ is either alpha-parallel or beta-parallel. 

We also focus our attention on a special 
class of complete systems, which we call \emph{meridional $\sigma$-basic systems of Heegaard diagrams}.  Such a basic system is constructed via the following procedure. We begin with a Heegaard link diagram $(\Sigma,\as,\bs,\ws,\zs)$ of $(S^3,L)$. If $K_i$ is a component of $L$, write $\lambda_i$ for the arc of $\scA$ for $K_i$. We assume that this Heegaard diagram is chosen so that the basepoints $w_i,z_i$ on  $K_i$ are separated by a single alpha curve $\a_i^s$ (if $\lambda_i$ is beta-parallel) or a single beta curve $\b_i^s$ (if $\lambda_i$ is alpha-parallel). We call the $\a_i^s$ and $\b_i^s$ the \emph{special meridional curves}. We assume that the arc $\lambda_i$ is embedded in $\Sigma$, and $\lambda_i$ is disjoint from all of the attaching curves except for the special meridional alpha or beta curve of $K_i$.  See Figure~\ref{cor:2}.

Suppose $\lambda_i$ is beta-parallel. We consider the component $A_i\subset \Sigma\setminus \as$ which contains $w_i$ and $z_i$. If we glue the two boundary components of $A_i$ corresponding to $\a_i^s$, we obtain a torus with many disks removed (corresponding to other alpha curves). We write $\as_0\subset \as$ and $\bs_0\subset \bs$ for the curves which are not special meridians of any component. See Figure~\ref{cor:2}.

\begin{figure}[ht]
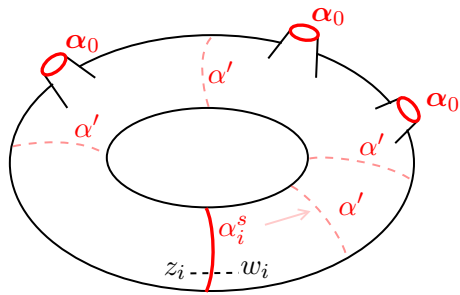
\caption{The region $A_i\subset \Sigma$ in a meridional $\sigma$-basic system. The arc $\lambda_i$ is the dashed arc which connects $z_i$ and $w_i$. The dashed curves labeled $\a'$ denote the successive replacements of $\a^s_i$ in the basic system.}
\label{cor:2}
\end{figure}

Given such a diagram $(\Sigma,\as,\bs,\ws,\zs)$, we construct two hyperboxes of Lagrangians $\cL_{\a}$ and $\cL_{\b}$, as follows. We pick surface diffeomorphisms $\phi_i\colon \Sigma\to \Sigma$, for $i\in \{1,\dots, n\}$. We assume that $\phi_i$ is supported in a neighborhood of the arc for $K_i$, and moves $z_i$ to $w_i$ along this arc. The hyperbox $\cL_{\a}$ has one axis direction for each component of $L$ which has a beta parallel arc in $\scA$. Similarly $\cL_{\b}$ has one axis direction for each component of $L$ which has an alpha parallel arc.

For the sake of concreteness, assume that $K_1,\dots, K_j$ are decorated with beta parallel arcs, and $K_{j+1},\dots, K_{\ell}$ are decorated with alpha parallel arcs.

We will construct a hyperbox $\cL_{\a}$ of some size $\ve{d}=(d_1,\dots, d_j)$. If $(\nu_1,\dots, \nu_j)\in \bE_j$, then  we assume
\[
\as_{(\nu_1 d_1,\dots, \nu_j d_j)}= \left(\phi_1^{1-\nu_1}\circ\cdots\circ \phi_j^{1-\nu_j}\right)(\as).
\]
That is, the attaching curves which are in the corners of $\cL_\a$ are obtained by applying some of the surface diffeomorphisms $\phi_i$ to $\as$. More generally, we assume that for all $\veps_1,\dots, \veps_{i-1},\veps_{i+1}, \dots, \veps_j$, we have
  \begin{equation}
  \phi_i(\as_{(\veps_1,\dots, d_i,\dots, \veps_j)})=\as_{(\veps_1,\dots, 0,\dots, \veps_j)}.
  \label{eq:phi_i-move-curve}
  \end{equation}

We form the $K_i$-direction of $\cL_{\a}$ as follows. Note that the surface isotopy $\phi_i$ only changes the special meridian $\a_i^s$. The subsequent steps in the $K_i$-direction correspond to moving $\a_i^s$ in the region $A_i\subset \Sigma\setminus \as_0$, while avoiding the basepoint $w_i$, so that $\a_i^s$ returns to its original position. This is achieved by a sequence of isotopies and handleslides of $\a_i^s$ across curves in $\as_0$. See Figure~\ref{cor:2}. Since all of the $A_i$ have disjoint interiors, we may do this independently for each link component of $L$ which has a beta-parallel arc. Each of the curves of $\as_0$ is stationary in this construction. To form the hyperbox $\cL_{\a}$, we perform small Hamiltonian translations to each curve of $\as_0$ to achieve admissibility.  The construction of $\cL_{\b}$ is similar.

In a complete system, we also require the $\cL_{\a}$ and $\cL_{\b}$ satisfy the following:
\begin{enumerate}
\item If $\veps<\veps'$ are in $\bE(\ve{d})$ and $|\veps'-\veps|_{L^1}=1$, then $\Theta_{\veps',\veps}$ is a cycle in $\ve{\CF}^-(\as_{\veps'},\as_{\veps})$ which represents the top degree of homology. If $|\veps'-\veps|_{L^{\infty}}=1$, then we require the chain $\Theta_{\veps',\veps}$ to have Maslov grading $|\veps'-\veps|_{L^1}-1$ relative to the top degree generator of homology.
\item As an extension of Equation~\eqref{eq:phi_i-move-curve}, we require that any chains on opposite faces of $\cL_{\a}$ coincide in the following sense. Suppose $\Theta_{\veps',\veps}\in \ve{\CF}^-(\as_{\veps'},\as_{\veps})$ is the chosen chain of $\cL_{\a}$ and 
\[
\veps=(\veps_1,\dots, 0, \dots, \veps_j)\qquad \text{and} \quad \veps'=(\veps_1',\dots, 0,\dots,\veps_j').
\]
Let 
\[
\hat{\veps}=(\veps_1,\dots, d_i, \dots, \veps_j)\qquad \text{and} \quad \hat{\veps}'=(\veps_1',\dots, d_i,\dots,\veps_j')
\]
and let $\Theta_{\hat{\veps}',\hat{\veps}}$ be the corresponding morphism in $\cL_{\a}$. We assume that $\Theta_{\veps',\veps}$ and $\Theta_{\hat{\veps}',\hat{\veps}}$ are equal under the canonical isomorphism
\[
(\phi_i)_*\colon \ve{\CF}^-(\as_{\hat{\veps}'},\as_{\hat{\veps}})\to \ve{\CF}^-(\as_{\veps'},\as_{\veps}).
\]
\item \label{La-Alexander-filtration} If $\veps<\veps'$, $|\veps'-\veps|_{L^\infty}=1$ and
\[
\veps=(\veps_1,\dots, d_i, \dots, \veps_j)\qquad \text{and} \quad \veps'=(\veps_1',\dots, d_i,\dots,\veps_j')
\]
then we require $A_i(\Theta_{\veps',\veps})\le 0$, where $A_i$ denotes the Alexander filtration induced by viewing $(\Sigma,\as_{\veps'},\as_{\veps},\ws,z_i)$ as a Heegaard diagram for an unknot in a connected sum of $S^1\times S^2$'s with extra free basepoints.
\end{enumerate}

 In requirement ~\eqref{La-Alexander-filtration}, we recall that if $\xs$ and $\ys$ are intersection points, the difference of their Alexander filtration is $A_i(\xs)-A_i(\ys)=n_{z_i}(\phi)-n_{w_i}(\phi)$ for a disk $\phi$ from $\xs$ to $\ys$. We extend this filtration to all multiples of intersection points by declaring the variable $U_i$ to decrease filtration level by $1$, and delcaring the other $U_k$ to preserve filtration level. See \cite{MOIntegerSurgery}*{Section~3.2} for more details.

\begin{rem}  Note that condition~\eqref{La-Alexander-filtration} is automatically satisfied for meridional systems since on the relevant faces of $\cL_{\a}$, the basepoints $w_i$ and $z_i$ are in the same region of the Heegaard diagram, so all intersection points have Alexander filtration $0$. 
\end{rem}

By pairing $\cL_{\a}$ and $\cL_{\b}$ together, we obtain an $\ell$-dimensional hyperbox of chain complexes $\ve{\CF}^-(\cL_{\a},\cL_{\b},\ws)$. Each of the constituent complexes are freely generated over $\bF\llsquare U_1,\dots, U_\ell \rrsquare$. To each axis direction, we add an additional layer to $\ve{\CF}^-(\cL_{\a},\cL_{\b})$ which consists of the tautological diffeomorphism maps $(\phi_i)_*$. We compress to obtain an $\ell$-dimensional hypercube of chain complexes, which we will denote by $(\scC_{\veps}, \scD_{\veps,\veps'})_{\veps\in \bE_\ell}$. We may view each $\scC_{\veps}$ as being the Floer complex obtained from $(\Sigma,\as,\bs,\ws,\zs)$ by keeping one basepoint from each link component, as follows. If $\veps\in \bE_\ell$, define the submodule 
\[
N_{\veps}\subset \cCFL(L)
\]
to be generated by $(\scV_{\veps_i}-1)\cdot \cCFL(L)$, ranging over $i$ such that $\veps_i=1$, as well as $(\scU_{\veps_i}-1)\cdot \cCFL(L)$, ranging over $i$ such that $\veps_i=0$. Then 
\[
\scC_{\veps}\iso \cCFL(L)/N_{\veps}.
\]

The link surgery hypercube $\cC_{\Lambda}(L)$ can be constructed as follows. If $K_i\subset L$ is a link component (oriented positively), then  $\Phi^{+K_i}$ is defined to be the canonical map for localizing at $\scV_i$. If $\vec{M}\subset L$ is a sublink, all of whose components are oriented oppositely to $L$, and $\veps',\veps\in \bE_\ell$ are points such that $\veps'-\veps$ is the indicator function for $M$, then the map $\Phi^{\vec{M}}$ is defined to be the unique map which reduces to $\scD_{\veps,\veps'}$ after quotienting the domain and codomain of $\Phi^{\vec{M}}$ by $N_{\veps}$ and $N_{\veps'}$, respectively, and which satisfies the Alexander grading property in Equation~\eqref{eq:grading-link-surgery}. One sets $\Phi^{\vec{M}}=0$ if $\vec{M}$ has more than one component and $\vec{M}$ has a component oriented coherently with $L$.

\section{The bordered perspective on the link surgery formula}
\label{sec:bordered}
In this section, we  recall the bordered perspective \cite{ZemBordered} on the knot and link surgery formulas \cite{OSKnots} \cite{MOIntegerSurgery}. We also prove several important new properties.

\subsection{Linear topological spaces}

In this section, we recall some preliminaries about completions. These are used in the module categories from \cite{ZemBordered}*{Section~6}. We refer the reader to \cite{AtiyahMacdonald}*{Section~10} for some basic background on completions of modules.  

A \emph{linear topological vector space} $\cX$ consists of a vector space equipped with a topology such that the following hold:
\begin{enumerate}
\item There is a basis of open sets centered at $0$ consisting of subspaces.
\item The addition function $\cX\times \cX\to \cX$ is continuous.
\end{enumerate}

Such a topology may be specified by picking a decreasing filtration $(\cX_{\a})_{\a\in A}$ of subspaces of $\cX$, indexed by some directed, partially ordered set $A$. Every linear topological space can be expressed this way, for example by setting $A$ to be the set of open linear subspaces of $\cX$, ordered by reverse inclusion.

 If $\cX$ has filtration $(\cX_{\a})_{\a\in A}$, as above, then the \emph{completion} of $\cX$ is the inverse limit 
\[
\ve{\cX}=\varprojlim_{\a\in A} \cX/\cX_{\a}.
\]

If $R$ is a ring, a \emph{linear topological $R$-module} is similar, except we require a basis at $0$ to consist of $R$-submodules.

If $\cX$ and $\cY$ are linear topological spaces, we will define a \emph{linear topological morphism} from $\cX$ to $\cY$ to be a continuous linear map from $\ve{\cX}$ to $\ve{\cY}$ (i.e. a continuous map between completions). 

Given two linear topological vector spaces (or linear topological $R$-modules) $\cX$ and $\cY$, there are several ways to topologize the tensor product. This phenomenon parallels notions from functional analysis, where the tensor product of two Banach spaces possesses many different Banach space structures. See \cite{GrothendieckNuclear}.  We consider the following topologies:
\begin{enumerate}
\item $\cX\otimes^! \cY$ (the \emph{standard tensor product}): A subspace $E\subset \cX\otimes^! \cY$ is open if and only if there are open subspaces $U\subset \cX$ and $V\subset \cY$ so that $U\otimes \cY$ and $\cX\otimes V$ are contained in $E$.
\item $\cX\vecotimes \cY$ (the \emph{chiral tensor product}): A subspace $E\subset \cX\vecotimes \cY$ is open if and only there is some open subspace $U\subset \cX$ such that $U\otimes \cY\subset E$, and for all $x\in X$ there is an open subspace $V_x\subset \cY$ so that $x\otimes V_x\subset E$.
\end{enumerate}
The above are described by Beilinson \cite{Beilinson-Tensors}. See Positelski's work \cite{Positelski-Linear} for a helpful introduction. The standard tensor product predates Beilinson's work and coincides with well-known constructions. The chiral tensor product is due to Beilinson. Both $\otimes^!$ and $\vecotimes$ are individually associative (with tensor products of the same type), but not mutually associative.

\subsection{The algebra $\cK$}
\label{sec:surgery-algebra}

  The author described in \cite{ZemBordered} the following associative algebra $\cK$, called the \emph{surgery algebra}. The algebra $\cK$ is an algebra over the ring of two idempotents $\ve{I}_0\oplus \ve{I}_1$, where $\ve{I}_{\veps}\iso \bF$ for $\veps\in \{0,1\}$. We set
\[
\ve{I}_0\cdot \cK\cdot \ve{I}_0=\bF[\scU,\scV]\quad\text{and} \quad \ve{I}_1\cdot \cK\cdot \ve{I}_1=\bF[\scU,\scV,\scV^{-1}].
\]
Furthermore, $\ve{I}_0\cdot \cK \cdot \ve{I}_1=0$. Finally, $\ve{I}_1\cdot \cK\cdot \ve{I}_0$ has two special algebra elements, $\sigma$ and $\tau$, which are subject to the relations
\[
\sigma\scU=\scU  \sigma, \quad \sigma  \scV=\scV \sigma,\quad \tau \scU=\scV^{-1} \tau,\quad \text{and} \quad \tau  \scV=\scU^2 \scV \tau,
\]
where $U=\scU\scV$.

It is sometimes helpful to consider the two algebra homomorphisms 
\[
\phi^\sigma,\phi^\tau\colon \bF[\scU,\scV]\to \bF[\scU,\scV,\scV^{-1}]
\]
 where $\phi^\sigma$ is the inclusion given by localizing at $\scV$, and $\phi^\tau$ satisfies $\phi^\tau(\scU)=\scV^{-1}$ and $\phi^\tau(\scV)=U \scV$. Then the relations for $\cK$ become
\[
\sigma\cdot a=\phi^\sigma(a)\cdot \sigma\quad \text{and} \quad \tau\cdot a=\phi^\tau(a)\cdot \tau,
\]
whenever $a\in \ve{I}_0\cdot \cK\cdot \ve{I}_0$. 

\begin{rem} There is a more symmetric description of the algebra $\cK$. If we view $\ve{I}_1\cdot \cK\cdot \ve{I}_1$ as being $\bF[U,T,T^{-1}]$ (where $U=\scU\scV$ and $T=\scV$) and then the relations of the algebra become
\[
\sigma \scU=UT^{-1} \sigma, \quad \sigma \scV=T \sigma,\quad \tau \scU=T^{-1} \tau\quad \text{and} \quad \tau \scV=UT \tau. 
\]
\end{rem}

\subsection{Modules over $\cK$}

As described in \cite{ZemBordered}*{Section~6.2}, the surgery algebra $\cK$ has a natural filtration consisting of the following subspaces:

\begin{define}
\label{def:knot-topology}
Suppose that $n\in \N$ is fixed. We define $J_n\subset \cK$ to be the span of following set of generators:
\begin{enumerate}
\item In $\ve{I}_0\cdot \cK\cdot \ve{I}_0$, the generators $\scU^i\scV^j$, for $i\ge n$ or $j\ge n$ (i.e.  $\max(i,j)\ge n$).
\item In $\ve{I}_1\cdot \cK\cdot \ve{I}_0$, the generators $\scU^i\scV^j \sigma$ for $i\ge n$ or $j\ge n$.
\item In $\ve{I}_1\cdot\cK\cdot \ve{I}_0$, the generators $\scU^i\scV^j\tau$ for 
$j\le 2i-n$ or $i\ge n$.
\item In $\ve{I}_1\cdot\cK\cdot \ve{I}_1$, the generators $\scU^i\scV^j$ where $i\ge n$.
\end{enumerate}
\end{define}

The subspaces $J_n$ are right ideals (i.e. $J_n\cdot \cK\subset J_n$). The filtration by the subspaces $J_n$ induces a linear topology on the algebra $\cK$.

In \cite{ZemBordered}*{Proposition~6.4} the author proves that multiplication is continuous as a map
\[
\mu_2\colon \cK\vecotimes_{\mathrm{I}} \cK\to \cK.
\]
In the terminology of Beilinson \cite{Beilinson-Tensors}, $\cK$ is a \emph{linear topological chiral algebra}.

\begin{rem} The map $\mu_2$ is not continuous as a map from $\cK\otimes^!_{\mathrm{I}} \cK$ to $\cK$ (i.e. using the standard tensor product topology). To see this, observe that $\scV^{-i}\otimes \scV^i\sigma\to 0$ in $\cK\otimes^!_{\mathrm{I}} \cK$, while $\mu_2(\scV^{-i}\otimes \scV^i\sigma)=\sigma\not\to 0$. 
\end{rem}

\begin{define}
\label{def:Alexander-type-D}
A \emph{type-$D$ Alexander module} over $\cK$ consists of a linear topological $\ve{I}$-module $\cX$ equipped with a linear topological morphism (i.e. a continuous linear map between completions) 
\[
\delta^1\colon \cX\to \cX\vecotimes \cK,
\]
 such that $(\id\otimes \mu_2)\circ(\delta^1\otimes \id_{\cK})\circ \delta^1=0$.
\end{define}

In \cite{ZemBordered}*{Section~6}, the author also describes the categories of type-$A$ and $DA$ Alexander modules. A type-$A$ Alexander module $(\cX, m_j)$ consists of a linear topological right $\ve{I}$-module $\cX$, equipped with linear topological morphisms
\[
m_{j+1}\colon \cK\vecotimes_{\mathrm{I}}\cdots \vecotimes_{\mathrm{I}} \cK\vecotimes_{\mathrm{I}} \cX\to \cX
\]
which satisfy the $A_\infty$-module relations. Type-$DA$ Alexander modules are similar. 

Lipshitz, Ozsv\'{a}th and Thurston's box tensor product operation \cite{LOTBordered}*{Section~2.4} can also be defined on Alexander modules.  If $\cX^{\cK}$ and ${}_{\cK} \cY$ are two Alexander modules, we may form their box tensor product
\[
\cX^{\cK}\hatbox {}_{\cK} \cY.
\]
This is a chain complex whose underlying linear topological space is
\[
\cX\vecotimes_{\mathrm{I}} \cY,
\]
and whose differential is given by the composition of the two maps
\[
\begin{tikzcd}
\cX\vecotimes \cY 
\ar[r, "\delta\otimes \id"]
&
\prod_{n\ge 0} (\cX\vecotimes  \underbrace{\cK\vecotimes \cdots \vecotimes \cK}_n \vecotimes \cY)\ar[r, "\id\otimes m_*"]
& \cX\vecotimes \cY.
\end{tikzcd}
\]
Here $\delta$ is the sum of all iterates of $\delta^1$ (the 0-th iterate being the identity). As in \cite{LOTBordered}, we additionally need to assume some boundedness assumptions on the structure maps of either $\cX^{\cK}$ or ${}_{\cK} \cY$ for the above composition to be well-defined. For our purposes, we will usually assume that $m_j=0$ on $\cY$ for all $j\gg 0$.

 Box tensor products of Alexander $DA$ modules may also be defined similarly.

\subsection{Surgery formulas and $\cK$}

The author showed how to view the knot and link surgery formulas of Ozsv\'{a}th--Szab\'{o} \cite{OSIntegerSurgeries} \cite{OSRationalSurgeries} and Manolescu--Ozsv\'{a}th \cite{MOIntegerSurgery} as type-$D$ and $A$ modules over the algebra $\cK$. To a knot $K$ in a rational homology 3-sphere $Y$, equipped with Morse framing $\lambda$, the mapping cone formula of Ozsv\'{a}th--Szab\'{o} may naturally be encoded using either a type-$D$ module over $\cK$, denoted $\cX_{\lambda}(Y,K)^{\cK}$, or a type-$A$ module over $\cK$, denoted ${}_{\cK} \cX_{\lambda}(Y,K)$. Similarly, if $L\subset S^3$ is a link with framing $\Lambda$, the link surgery formula of Manolescu and Ozsv\'{a}th \cite{MOIntegerSurgery} can be naturally encoded using a type-$D$ module $\cX_{\Lambda}(L)^{\cL_{\ell}}$ over the tensor product algebra $\cL_\ell=\cK^{\otimes_{\bF} \ell}$, where $\ell=|L|$.

We describe the correspondence for the case of the knot surgery formula for a knot $K\subset S^3$. We recall from Section~\ref{sec:background-surgery-formulas} that Ozsv\'{a}th and Szab\'{o}'s mapping cone complex is an appropriate completion of the mapping cone
\[
\Cone\left(v+h_n\colon \cCFK(K)\to \scV^{-1} \cCFK(K)\right).
\]
Here,  $v$ and $h_n$ are two $\bF[U]$-equivariant maps, where $U$ acts by $\scU\scV$.

 The type-$D$ module $\cX_{n}(K)^{\cK}$ is defined as follows. Let $\xs_1,\dots, \xs_n$ be a free $\bF[\scU,\scV]$-basis of $\cCFK(K)$. Then $\cX_n(K)\cdot \ve{I}_0$ is generated over $\bF$ by copies of the basis elements $\xs_1^0,\dots, \xs_n^0$, and similarly $\cX_n(K)\cdot \ve{I}_1$ is generated by elements $\xs_1^1,\dots, \xs_n^1$. The structure map
\[
\delta^1\colon \cX_n(K)\to \cX_n(K)\otimes_{\mathrm{I}} \cK
\]
has three types of summands: those arising from $\d$, those arising from $v$, and those arising from $h_n$, as follows. Write $\d$ for the differential on $\cCFK(K)$. If $\d\xs_i$ has a summand of $ \ys_j\cdot \scU^n\scV^m$, then $\delta^1(\xs_i^\veps)$ has a summand of $\ys_j^\veps\otimes \scU^n \scV^m$, for $\veps\in \{0,1\}$. Each $\delta^1(\xs^0_i)$ has a summand of the form $\xs^1_i\otimes \sigma$. Finally, if $h_n(\xs_i)$ has a summand of $ \ys_j\cdot\scU^n\scV^m$, then $\delta^1(\xs_i^0)$ has a summand of $\ys_j^1\otimes \scU^n \scV^m \cdot\tau$. It is verified in \cite{ZemBordered}*{Lemma~8.9} that $\cX_n(K)^{\cK}$ is a type-$D$ module over $\cK$.

For the description of the  link surgery complex $\cC_{\Lambda}(L)$ of an integrally framed link $L\subset S^3$ in terms of type-$D$ modules over the algebra $\cL_\ell$, we refer the reader to \cite{ZemBordered}*{Section~8}. We note that in general, the homotopy type of the link surgery type-$D$ module depends non-trivially on the system of arcs $\scA$. For example, the type-$D$ link surgery module of the Hopf link depends non-trivially on $\scA$;  see \cite{ZemBordered}*{Section~16}. We write $\cX_{\Lambda}(L,\scA)^{\cL}$ for the type-$D$ module constructed with $\scA$.

There is also a bimodule ${}_{\cK} \cT^{\cK}$ whose effect is changing an alpha-parallel arc to a beta-parallel arc, or vice-versa. See \cite{ZemBordered}*{Section~14}. This bimodule is related to the Dehn twist diffeomorphism on knot Floer homology discovered by Sarkar \cite{SarkarMovingBasepoints} and further studied by the author \cite{ZemQuasi}.

\subsection{Type-$A$ and $D$ modules for solid tori}

We now review the type-$A$ and type-$D$ modules of integrally framed solid tori (by which we mean the complements of integrally framed unknots in $S^3$).

If $n\in \Z$, we define the type-$A$ and type-$D$ modules of the $n$-framed solid torus, denoted ${}_{\cK} \cD_n$ and $\cD_n^{\cK}$, as follows. We begin with ${}_{\cK} \cD_n$. We set 
\[
\ve{I}_0\cdot \cD_n=\bF\llsquare \scU,\scV\rrsquare\quad \text{and}\quad  \ve{I}_1\cdot \cD_n=\bF\llsquare \scU,\scV,\scV^{-1}\rrsquare.
\]
The action of $\ve{I}_{\veps}\cdot \cK\cdot \ve{I}_{\veps}$ is ordinary polynomial multiplication. If $\ve{x}\in \ve{I}_0\cdot \cD_n$,  we set $m_2(\sigma, \ve{x})=\phi^\sigma(\ve{x})$ and $m_2(\tau, \ve{x})=\scV^n\cdot \phi^\tau(\ve{x})$, where
\[
\phi^\sigma,\phi^\tau\colon \bF[\scU,\scV]\to \bF[\scU,\scV,\scV^{-1}]
\]
are the following maps. The map $\phi^\sigma$ is the inclusion from localizing at $\scV$. The map $\phi^\tau$ is given by $\phi^\tau(\scU^i\scV^j)=\scU^j \scV^{2j-i}$.

The type-$D$ module $\cD_n^{\cK}$ is as follows. As a right $\ve{I}$-module, we set $\cD_n\iso \ve{I}$, viewed as being generated by $i_0\in \ve{I}_0$ and $i_1\in \ve{I}_1$. The structure map is given by the formula
\[
\delta^1(i_0)=i_1\otimes (\sigma+\scV^n \tau).
\]

\subsection{The merge and type-$A$ identity bimodules}
\label{sec:merge}

In \cite{ZemBordered}*{Section~8}, the author defined the \emph{merge} bimodule and the \emph{type-$A$ identity} bimodule. We recall the definition of these modules presently.

We begin with the merge module ${}_{\cK| \cK} M^{\cK}$. Ignoring completions, the merge module is a $DA$-bimodule over $(\cK\otimes_{\bF} \cK, \cK)$. However, the structure map $\delta_3^1$ of the merge module is not continuous as a map from $(\cK\otimes_{\bF}^! \cK)\vecotimes (\cK\otimes_{\bF}^! \cK)\vecotimes M$ to $M\vecotimes \cK$. Instead, the merge module is a \emph{split Alexander} module in the terminology of \cite{ZemBordered}*{Section~6.4}. This is weaker than being an Alexander module. This condition means that the map $\delta_3^1$ is continuous for a finer topology than the topology used to define an Alexander module over $\cK\otimes_{\bF} \cK$. The split Alexander property is sufficient for taking the box tensor product of ${}_{\cK|\cK} M^{\cK}$ with a pair of type-$D$ modules $\cX^{\cK}$ and $\cY^\cK$, but not for general type-$D$ modules over $\cK\otimes_{\bF}^! \cK=\cL_2$. 

 As an $(\ve{I}\otimes \ve{I},\ve{I})$-module, $M$ is isomorphic to $\ve{I}$. The right action is the obvious one. The left action is given by $(i_{\veps}\otimes i_{\nu})\cdot i=i_{\veps}\cdot i_{\nu}\cdot i$.
The map $\delta_2^1$ is defined as follows. If $a,b\in \ve{I}_{0} \cdot \cK \cdot \ve{I}_0$, then we set $\delta_2^1(a\otimes b, i_0)=i_0\otimes a\cdot b$. We use the same formula if  $a,b\in \ve{I}_1\cdot \cK \cdot \ve{I}_1$. If $a$ and $b$ are in other idempotents, then we set $\delta_2^1(a\otimes b,i_\veps)=0$. 

 Additionally, there is a $\delta_3^1$ term, determined by the equations
\[
\delta_3^1(\tau\otimes 1, 1\otimes \tau, i_0)=i_0\otimes\tau,\quad \quad \delta_3^1(\sigma\otimes 1, 1\otimes \sigma, i_0)=i_1\otimes \sigma \quad \text{and}
\]
\[
\delta_3^1(1\otimes \tau,\tau\otimes 1, i_0)=\delta_3^1(1\otimes \sigma,\sigma\otimes 1, i_0)=0.
\]

Using the merge module, we define the type-$A$ identity bimodule
\[
{}_{\cK|\cK} [\bI^{\Supset}]:={}_{\cK| \cK} M^{\cK}\hatbox {}_{\cK} \cD_0.
\]
The type-$A$ identity modules relates the type-$D$ and type-$A$ modules of a knot complement $K$ via the formula
\[
\cX_n(K)^\cK \hatbox {}_{\cK|\cK}  [\bI^{\Supset}]\iso {}_{\cK} \cX_n(K).
\]

\begin{rem} In our definition of the completed box tensor product $\hatbox$, we use the chiral tensor product $\vecotimes$. However, since $M$ is a finite dimensional vector space, the resulting topology on $[\bI^{\Supset}]$ coincides with the topology on ${}_{\cK}\cD_0$. Similarly the topology on ${}_{\cK} \cX_n(K)$ coincides with the topology on  $\cX_n(K)^{\cK}\hatbox {}_{\cK} \cD_0$.
\end{rem}

\subsection{A pairing theorem}

 Suppose that $L_1$ and $L_2$ are two links in $S^3$. The author \cite{ZemBordered}*{Section~12}  gave several descriptions of the link surgery formula for $L_1\#L_2$ in terms of the link surgery formulas for $L_1$ and $L_2$. We refer to these connected sum formulas \emph{pairing theorems} since topologically performing $\lambda_1+\lambda_2$ surgery on $K_1\# K_2\subset S^3$ is the same as gluing the complements of $K_1$ and $K_2$ via the orientation reversing diffeomorphism which sends the meridian $\mu_1$ to $\mu_2$ and the longitude $\lambda_1$ of $K_1$ to the longitude $-\lambda_2$ of $K_2$.  See \cite{ZemBordered}*{Section~12} for more details on these pairing theorems.

We begin by describing the pairing theorem on the level of the link surgery complexes. Subsequently, we will describe the connected sum formula in terms of type-$D$ modules. Write $\cC_{\Lambda_1}(L_1)=(\cC_{\veps}(L_1),d_{\veps,\veps'})$ and $\cC_{\Lambda_2}(L_2)=(\cC_{\nu}(L_2),\delta_{\nu,\nu'})$ for the link surgery formulas of links $L_1$ and $L_2$, with integral framings $\Lambda_1$ and $\Lambda_2$. Write $K_1$ and $K_2$ for the distinguished components of $L_1$ and $L_2$ (along which we take the connected sum). We assume that the link surgery complexes are computed with systems of arcs $\scA_1$ and $\scA_2$ so that $K_1$ has an alpha-parallel arc, and $K_2$ has a beta-parallel arc. We write $\scA_{1\# 2}$ for the system of arcs on $L_1 \# L_2$ which coincides with $\scA_1$ and $\scA_2$ away from $K_1\# K_2$, and which has an arc for $K_1\# K_2$ consisting of the co-core of the connected sum band. See Figure~\ref{fig:4}.

\begin{figure}[ht]
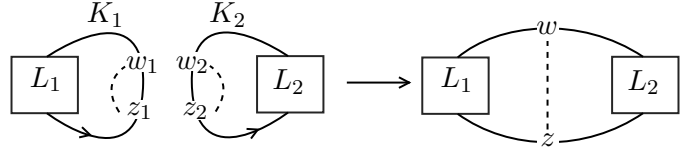
\caption{Left: arc systems $\scA_1$ and $\scA_2$ so that components $K_1\subset L_1$ and $K_2\subset L_2$ have arcs which are beta-parallel and alpha-parallel. Right: The arc system $\scA_{1\#2}$ on $L_1\# L_2$.}
\label{fig:4}
\end{figure}

 If $\veps\in \{0,1\}$, write $\cC^{(*,\veps)}(L_1)$ for the codimension one subcube of $\cC_{\Lambda_1}(L_1,\scA_1)$ consisting of complexes which have $K_1$-component $\veps$. Define $\cC^{(*,\veps)}(L_2)$ similarly. We may view $\cC_{\Lambda_1}(L_1,\scA_1)$ as a mapping cone from $\cC^{(*,0)}(L_1)$ to $\cC^{(*,1)}(L_1)$ as follows:
 \[
 \cC_{\Lambda_1}(L_1,\scA_1)=\Cone\big( \begin{tikzcd}[column sep=1.5cm] \cC^{(*,0)}(L_1)\ar[r, "F^{K_1}+F^{-K_1}"]& \cC^{(*,1)}(L_1)
 \end{tikzcd}\big).
 \]
In the above, and $F^{K_1}$ is the sum of the hypercube maps $\Phi^{\vec{M}}$ of $\cC_{\Lambda_1}(L_1)$ ranging over all oriented sublinks $\vec{M}\subset L_1$ such that $+K_1\subset \vec{M}$. Similarly $F^{-K_1}$ is the sum of the hypercube maps such that $-K_1\subset \vec{M}$. Each $\cC^{(*,\veps)}(L_1)$ has an internal differential consisting of the sum of the hypercube maps for sublinks $\vec{M}$ which do not contain $\pm K_1$.  We may similarly view $\cC_{\Lambda_2}(L_2)$ as a mapping cone from $\cC^{(*,0)}(L_2)$ to $\cC^{(*,1)}(L_2)$. 

The pairing theorem is the following:
\begin{thm}[\cite{ZemBordered}*{Theorem~12.1}]
\label{thm:pairing-main}
With respect to the above notation, the link surgery complex $\cC_{\Lambda_1+\Lambda_2}(L_1\# L_2,\scA_{1\# 2})$ is homotopy equivalent to
\[
 \begin{tikzcd}[column sep=3.4cm] \cC^{(*,0)}_1\otimes_{\bF[\scU,\scV]} \cC^{(*,0)}_2\ar[r, "F^{K_1}| F^{K_2}+F^{-K_1}| F^{-K_2}"] & \cC^{(*,1)}_1\otimes_{\bF[\scU,\scV,\scV^{-1}]} \cC^{(*,1)}_2.
 \end{tikzcd}
\]
In the above $\cC_i^{(*,\veps)}$ denotes $\cC^{(*,\veps)}(L_i)$ and $\Lambda_1+\Lambda_2$ is obtained by summing the framing on $K_1$ and $K_2$, and using the  framings from $L_1$ and $L_2$ on the other components. Also, the differential on $\cC^{(*,\veps)}(L_1)\otimes \cC^{(*,\veps)}(L_2)$ is the ordinary differential (Leibniz rule) on the tensor product of two chain complexes. 
\end{thm}

Note that there is a stronger version of Theorem~\ref{thm:pairing-main} in terms of the type-$D$ modules. According to \cite{ZemBordered}*{Theorem~12.1}, there is a homotopy equivalence
\[
\begin{split}
&\cX_{\Lambda_1+\Lambda_2}(L_1\#L_2,\scA_{1\#2})^{\cL_{\ell_1+\ell_2-1}}\\
\simeq& \left(\cX_{\Lambda_1}(L_1,\scA_1)^{\cL_{\ell_1}} , \cX_{\Lambda_2}(L_2,\scA_2)^{\cL_{\ell_2}}\right)\hatbox {}_{\cK| \cK} M^{\cK}.
\end{split}
\]
Note that if we ignore completions, the  box tensor product in the above equation is obtained as follows: We first take the external tensor product of $\cX_{\Lambda_1}(L_1,\scA_1)^{\cL_{\ell_1}}$ and $ \cX_{\Lambda_2}(L_2,\scA_2)^{\cL_{\ell_2}}$, i.e. we take the tensor product of the modules over $\bF$, and use the Leibniz rule to form the differential $\delta^1\otimes \id\otimes 1_{\cL_{\ell_2}}+\id\otimes 1_{\cL_{\ell_1}}\otimes \delta^1$ (with tensor factors reordered in the obvious way). This yields a type-$D$ module over $\cL_{\ell_1+\ell_2}$. Next, we view ${}_{\cK|\cK} M^{\cK}$ as a type-$DA$ module over $(\cK\otimes_{\bF} \cK,\cK)$, and take the external tensor product of $M$ with the identity bimodule ${}_{\cL_{\ell_1+\ell_2-2}}[\bI]^{\cL_{\ell_1+\ell_2-2}}$ to get a $DA$-bimodule over $(\cL_{\ell_1+\ell_2},\cL_{\ell_1+\ell_2-1})$. Finally, we compute the box tensor product as normal \cite{LOTBordered}*{Section~2.4}.

\subsection{The pair-of-pants bimodules}

The connected sum formula in Theorem~\ref{thm:pairing-main} requires one of the components $K_1\subset L_1$ and $K_2\subset L_2$ to have an alpha-parallel arc, and the other to have a beta-parallel arc. The output arc in $\scA_{1\# 2}$ is neither alpha-parallel nor beta-parallel. For taking further connected sums with other knots, we need to change the arc $\scA_{1\# 2}$ so that the arc for $K_1\#K_2$ is either alpha or beta-parallel. A general formula for changing arcs is described in \cite{ZemBordered}*{Section~13.2}. 

 In our present case, taking the connected sum and then changing the arc for $K_1\# K_2$ can be encoded by one of two bimodules, which are similar to the merge module. We call these the \emph{pair-of-pants} bimodules,  and we denote them by ${}_{\cK|\cK} W_l^{\cK}$ and ${}_{\cK|\cK} W_r^{\cK}$. Their definitions will be recalled below.

Before we describe the bimodules $W_l$ and $W_r$, we will recall some motivation for the formula. Heegaard Floer homology is an invariant of based 3-manifolds. There is a fibration
\[
\Diff^+(Y,w)\to \Diff^+(Y)\to Y.
\]
The long exact sequence on homotopy groups gives a map $\pi_1(Y,w)\to \pi_0\Diff^+(Y,w),$ and therefore a loop  $\g$ in $Y$ based at $w$ induces an endomorphism of $\CF^-(Y,w)$. If $\g\in \pi_1(Y,w)$, we write $\phi_\gamma$ for the induced diffeomorphism. In \cite{ZemGraphTQFT}*{Theorem~D}, we computed the induced action on Heegaard Floer homology to be
\begin{equation}
(\phi_\gamma)_*\simeq \id+ \Phi_{w}\circ A_{\g},
\label{eq:basepoint-moving}
\end{equation}
where $A_\g$ denotes the action of $\Lambda^* H_1(Y)/\Tors$, and $\Phi_w$ denotes the algebraic basepoint action.

Next, we observe that if $\g_1\in H_1(Y_1)$ and $Y_2$ is some other 3-manifold, then  $\g_1\in H_1(Y_1\#Y_2)$ acts on $\CF^-(Y_1\# Y_2)$  by $A_{\g_1}\otimes\id$. The basepoint action on $\CF^-(Y_1\# Y_2)\iso \CF^-(Y_1)\otimes_{\bF[U]} CF^-(Y_2)$ is equal to
\[
\Phi_1\otimes \id+\id\otimes \Phi_2.
\]

Note that in the case of the arc system $\scA_{1\# 2}$, we may obtain an alpha parallel or beta parallel arc system by concatenating the arc for $K_1\# K_2$ with an arc that runs parallel to either $K_1$ or $K_2$. In the case when $L_1$ and $L_2$ both have a single component, it is possible to use Equation~\eqref{eq:basepoint-moving} to prove that the map $\Phi^{-K_1\#K_2}$ becomes the composition of $\Phi^{-K_1}|\Phi^{-K_2}$ with the basepoint moving map 
\[
\id+\scV^{-1} (\Phi_1|1+1|\Phi_2)\circ (A_{[K_1]}|1).
\]
(The factor of $\scV^{-1}$ is due to the fact that $\Phi_i$ shift the Alexander grading up by 1). 

In \cite{ZemBordered}*{Section 15.2}, we generalize the above discussion to the setting of the surgery hypercubes, and we describe the effect of changing $\scA_{1\# 2}$ to an alpha or beta parallel arc system. To state our result, we write $\cC_{\Lambda_i}(L_i,\scA_i)$ as mapping cones 
 \[
 \cC_{\Lambda_i}(L_i,\scA_i)=\Cone\left( \begin{tikzcd}[column sep=1.5cm] \cC^{(*,0)}(L_i)\ar[r, "F^{K_i}+F^{-K_i}"] & \cC^{(*,1)}(L_i)
 \end{tikzcd}\right).
 \]
 
 \begin{thm}[\cite{ZemBordered}*{Section~15.2}]
 \label{thm:pair-of-pants-on-hypercube}
Suppose that $L_1,L_2\subset S^3$ are two framed links with systems of arcs $\scA_1$ and $\scA_2$, respectively. Suppose also that we form $L_1\# L_2$ by taking the connected sum along components $K_1\subset L_1$ and $K_2\subset L_2$. Suppose that the arc for $K_1$ is alpha-parallel, and the arc for $K_2$ is beta-parallel. Let $\scA_l$ denote the system of arcs on the connected sum which coincides with $\scA_1$ and $\scA_2$ away from $K_1\# K_2$ and is alpha-parallel on $K_1\#K_2$. Then $\cC_{\Lambda_1+\Lambda_2}(L_1\# L_2, \scA_{l})$ may also be described as the mapping cone
\begin{equation}
 \begin{tikzcd}[column sep=5cm] \cC^{(*,0)}_1\otimes \cC^{(*,0)}_2\ar[r, "F^{K_1}| F^{K_2}+\Tw_{[K_1]}\circ (F^{-K_1}| F^{-K_2})"] & \cC^{(*,1)}_1\otimes \cC^{(*,1)}_2,
 \end{tikzcd}
 \label{eq:alternate-pairing-theorem}
\end{equation}
where
\[
\Tw_{[K_1]}=\id+  \scV^{-1}(\Phi_{w_1}|1+1|\Phi_{w_2})\circ (\cA_{[K_1]} | 1)
\]
and $\cC^{(*,\veps)}_i$ denotes $\cC^{(*,\veps)}(L_i)$.
\end{thm}

In the above theorem, the maps $\Phi_{w_i}$ denote the analogs of the basepoint actions for the hypercubes $\cC^{(*,1)}(L_i)$. These maps are defined similarly to basepoint actions on the ordinary link Floer complexes, except are defined in the setting of hypercubes. Note that they are morphisms of hypercubes, so will generally have non-trivial components of length greater than zero (in particular, $\Phi_{w_i}$ is not the same as the internal basepoint action on $\cCFL(L_i)$). The map $\cA_{[K_1]}$ is the hypercube homology action of the curve $K_1\subset \Sigma_1$.  See \cite{ZemBordered}*{Sections~13.2 and 13.3} for more detail on these constructions.

The map $\cA_{[K_1]}$ appearing in Equation~\eqref{eq:alternate-pairing-theorem} has another description, which is helpful from an algebraic perspective. According to \cite{ZemBordered}*{Lemma~13.28}, there is a chain homotopy
\begin{equation}
\cA_{[K_1]}\simeq \scU\Phi_{w_1}+\scV \Psi_{z_1}, \label{eq:homology-action-basepoint}
\end{equation}
where we view both maps as being endomorphisms of the hypercube $\cC^{(*,1)}(L_1)$. In the above, we are writing $\Phi_{w_1}$ and $\Psi_{z_1}$ for the algebraic basepoint action of the basepoints of $K_1$ on the hypercube $\cC^{(*,1)}(L_1)$.

We can translate Theorem~\ref{thm:pair-of-pants-on-hypercube} into a statement about bimodules over the surgery algebras. See \cite{ZemBordered}*{Section~15.2} for more details. We will define two modules ${}_{\cK|\cK}W_r^{\cK}$ and ${}_{\cK|\cK} W_l^{\cK}$ similarly to the merge module ${}_{\cK|\cK} M^{\cK}$ from Section~\ref{sec:merge}. We define them to have the same underlying vector space as the merge module. The module ${}_{\cK|\cK} W_r^{\cK}$ has $\delta_2^1$ and $\delta_3^1$ identical to the merge module. Additionally, there is a $\delta_5^1$, as follows:
\begin{equation}
\delta_5^1(a|b, a'|b', 1|\tau,\tau|1, i_0)=i_1\otimes \scV^{-1}\d_{\scU}(ab) a' \left(\scU \d_{\scU}+\scV \d_{\scV}\right)(b') \tau.  \label{eq:delta-5-1-merge}
\end{equation}
In the above, $\d_{\scU}$ and $\d_{\scV}$ denote the derivatives with respect to $\scU$ and $\scV$, respectively. The module ${}_{\cK|\cK}W_l^{\cK}$ is similar, but has the role of the two tensor factors switched. Theorem~\ref{thm:pair-of-pants-on-hypercube} has the following type-$D$ analog:

\begin{thm}[\cite{ZemBordered}*{Theorem~15.2}] Suppose that $L_1,L_2\subset S^3$ are two framed links with systems of arcs $\scA_1$ and $\scA_2$, respectively. Suppose also that we form $L_1\# L_2$ by taking the connected sum along components $K_1\subset L_1$ and $K_2\subset L_2$. Suppose that the arc for $K_1$ is alpha-parallel, and the arc for $K_2$ is beta-parallel. Let $\scA_l$ denote the system of arcs on the connected sum which coincides with $\scA_1$ and $\scA_2$ away from $K_1\# K_2$ and is alpha-parallel on $K_1\#K_2$. Let $\scA_r$ denote the analogous system of arcs on $L_1\#L_2$ which is instead beta-parallel on $K_1\# K_2$. Then
\[
\cX_{\Lambda_1+\Lambda_2}(L_1\#L_2,\scA_r)^{\cL_{\ell_1+\ell_2-1}}\simeq (\cX_{\Lambda_1}(L_1,\scA_1)^{\cL_{\ell_1}}, \cX_{\Lambda_2}(L_2,\scA_2)^{\cL_{\ell_2}})\hatbox {}_{\cK|\cK} W_{r}^{\cK},
\] 
and similarly if $\scA_{r}$ and $W_r$ are replaced by $\scA_{l}$ and $W_l$.
\end{thm}

\subsection{The Hopf link surgery complex}
\label{sec:Hopf-link}
We now recall the link surgery hypercube for the Hopf link, which was computed in \cite{ZemBordered}*{Section~16}. We recall the negative Hopf link has the following link Floer complex.
 \begin{equation}
\cCFL(H)\iso \begin{tikzcd}[labels=description,row sep=1cm, column sep=1cm] \ve{a}& \ve{b}\ar[l, "\scU_2"] \ar[d, "\scU_1"]\\
\ve{c}\ar[r, "\scV_2"] \ar[u, "\scV_1"]& \ve{d}.
\end{tikzcd}
\label{eq:Hopf-def}
\end{equation}

 There are two models for the Hopf link surgery hypercube, depending on the choice of arc system. The models are summarized in the following result:
 \begin{prop}[\cite{ZemBordered}*{Proposition~16.1}]
 \label{prop:Hopf-link-large-model} Write $H=K_1\cup K_2$ for the negative Hopf link, and suppose that $\Lambda=(\lambda_1,\lambda_2)$ is an integral framing $H$. Let $\scA$ be a system of arcs for $H$ where both components are alpha-parallel or both components are beta-parallel. The maps in the surgery complex $\cC_{\Lambda}(H,\scA)$ are, up to overall homotopy equivalence, as follows:
 \begin{enumerate}
  \item The map $\Phi^{K_1}$ is the canonical inclusion of localization. The map $\Phi^{-K_1}$ is given by the following formula:
  \[
 \Phi^{-K_1} =\begin{tikzcd}[row sep=0cm]
 \ve{a} \ar[r,mapsto]& \ve{d}\scV_1^{\lambda_1-1}\\
 \ve{b} \ar[r, mapsto]& 0\\
 \ve{c} \ar[r, mapsto]& \ve{b}\scV_1^{\lambda_1+1}+\ve{c}\scV_1^{\lambda_1}\scU_2\\
 \ve{d} \ar[r,mapsto]& \ve{d}\scV_1^{\lambda_1}\scU_2.
 \end{tikzcd}
  \]
 \item The map $\Phi^{K_2}$ is the canonical inclusion of localization, and $\Phi^{-K_2}$ is given by the following formula:
\[
\Phi^{-K_2}=\begin{tikzcd}[row sep=0cm]
\ve{a}\ar[r, mapsto] &\ve{a}\scU_1\scV_2^{\lambda_2}\\
 \ve{b}\ar[r, mapsto] &0\\
 \ve{c}\ar[r, mapsto] &\ve{b}\scV_2^{\lambda_2+1}+\ve{c}\scU_1\scV_2^{\lambda_2} \\
 \ve{d}\ar[r, mapsto]&\ve{a}\scV_2^{\lambda_2-1}.
 \end{tikzcd}
\]
 \item The length 2 map $\Phi^{-K_1\cup -K_2}$ is given by the following formula:
 \[
\Phi^{-K_1\cup -K_2}=\begin{tikzcd}[row sep=0cm]
 \ve{a}\ar[r, mapsto] &\ve{c}\scV_1^{\lambda_1-2}\scV_2^{\lambda_2-1}\\
 \ve{b}\ar[r, mapsto] &0\\
  \ve{c}\ar[r, mapsto] &\ve{d}\scV_1^{\lambda_1-1} \scV_2^{\lambda_2}\\
 \ve{d}\ar[r, mapsto]&\ve{c}\scV_1^{\lambda_1-1} \scV^{\lambda_2-2}_2.
 \end{tikzcd}
 \]
 The length 2 maps for other orientations of the Hopf link vanish.
 \end{enumerate}
 \end{prop}
 The above formulas are stated only for the values of the maps on the generators $\ve{a}$, $\ve{b}$, $\ve{c}$ and $\ve{d}$. Values on the rest of the complex are determined by the equivariance properties of the maps proven in \cite{ZemBordered}*{Lemma~7.7}.

  \begin{prop}[\cite{ZemBordered}*{Proposition~16.7}] Let $\bar\scA$ be a system of arcs on the Hopf link $H$ where one arc is alpha-parallel and the other arc is beta-parallel. The surgery complex for the Hopf link $\cC_{\Lambda}(H,\bar\scA)$ is identical to the one in Proposition~\ref{prop:Hopf-link-large-model}, except that we delete the term $\ve{c}\mapsto  \ve{d} \scV_1^{\lambda_1-1} \scV_2^{\lambda_2} $ from the expression in $\Phi^{-K_1\cup -K_2}$.
  \end{prop}
 
 The above descriptions of $\cC_{\Lambda}(H,\scA)$ and $\cC_{\Lambda}(H,\bar \scA)$ may be repackaged as type-$D$ modules over the algebra $\cK\otimes_{\bF} \cK$. See \cite{ZemBordered}*{Section~8.6}.
  We write  $\cH_{\Lambda}^{\cK\otimes \cK}$
  and $\bar \cH_{\Lambda}^{\cK\otimes \cK}$ for these two surgery complexes.
  
 It is convenient to consider the type-$DA$ versions of the Hopf link complexes, and we set
 \[
 {}_{\cK} \cH_{\Lambda}^{\cK}:= \cH_{\Lambda}^{\cK\otimes \cK}\hatbox {}_{\cK| \cK} [\bI^{\Supset}],
 \]
 where the tensor product is taken on a single algebra factor. We define ${}_{\cK} \bar{\cH}_{\Lambda}^{\cK}$ analogously.

\begin{rem}
 The Hopf link $DA$-bimodule is related to the dual knot surgery formula of Hedden--Levine \cite{HeddenLevineSurgery} and Eftekhary \cite{EftekharyDuals}. Indeed, tensoring $\cX_n(K)^{\cK}$ with ${}_{\cK} \cH_{\Lambda}^{\cK}$ has the effect of computing their dual knot surgery formula.
 \end{rem}

\section{Homological actions on the surgery complex}
\label{sec:homology-action}

In this section, we study several endomorphisms of the link surgery hypercube which are related to the standard $\Lambda^* (H_1(Y)/\Tors)$ action on $\HF^-(Y)$. In Section~\ref{sec:algebraic-action-link-surgery}, we study one endomorphism of the surgery cube which is obtained by summing over a subset of the structure maps. In Section~\ref{sec:diagrammatic-action}, we study an endomorphism of $\cC_{\Lambda}(L)$ induced by a closed curve $\g\subset \Sigma$. This morphism is computed by counting holomorphic polygons with certain weights. In Section~\ref{sec:relating-actions}, we relate these actions for meridional $\sigma$-basic systems. We call these the \emph{algebraic} and \emph{diagrammatic} actions, respectively.

In Section~\ref{sec:simplifying-connected-sum-1}, we describe an application of these results to simplify the connected sum formula in certain cases.

\subsection{A first homology action}
\label{sec:algebraic-action-link-surgery}

We now define an action of 
\[
\Lambda^* (H_1(S^3_{\Lambda}(L))/\Tors)
\]
 on $\cC_{\Lambda}(L).$ Write $H_1(S^3_{\Lambda}(L))=\Z^\ell/\im \Lambda$, where $\Z^\ell$ denotes the free abelian group generated by the meridians $\mu_i$ of the components of $L$.

 Define an endomorphism $F^{K_i}$ of  $\cC_{\Lambda}(L)$ to be the sum of all hypercube structure maps for oriented sublinks $\vec{M}\subset L$ which contain $+K_i$. Define $F^{-K_i}$ similarly. We define an endomorphism $\frA_{\mu_i}$ of $\cC_{\Lambda}(L)$ as
\[
\frA_{\mu_i}=F^{K_i}.
\]
We think of $[\mu_i]$ as the element $(0,\dots, 1, \dots,0)$ in $\Z^\ell/\im \Lambda\iso H_1(S^3_{\Lambda}(L))$.

 Note that $F^{K_i}\simeq F^{-K_i}$ as endomorphisms of $\cC_{\Lambda}(L)$, since a chain homotopy is given by projecting to the codimension 1 subcube with $K_i$-coordinate $0$.

 For a general $\g\in H_1(S^3_\Lambda(L))$, write $\g=a_1 \cdot \mu_1+\cdots
+ a_\ell \cdot \mu_\ell$ and define
\begin{equation}
\frA_{\g}:=a_1 \cdot \frA_{\mu_1}+\cdots +a_\ell\cdot \frA_{\mu_\ell}.\label{eq:hypercube-homology-action}
\end{equation}
We now prove that this definition gives a well-defined action of $\Lambda^* H_1(S^3_{\Lambda}(L))/\Tors$ on $\cC_{\Lambda}(L)$.

\begin{lem}\label{lem:algebraic-homology-action-link-surgery} Given $\g\in H_1(S^3_\Lambda(L))/\Tors$, the endomorphism $\frA_{\g}$ of $\cC_{\Lambda}(L)$ is well-defined up to $\bF\llsquare U_1,\dots, U_\ell\rrsquare $-equivariant chain homotopy.
\end{lem}
\begin{proof} We recall that $H_1(S^3_{\Lambda}(L))$ is isomorphic to $\Z^\ell/\im \Lambda$, where we view $\Lambda$ as the $\ell\times \ell$ symmetric matrix whose diagonal entries consist of the framings, and whose off diagonal entries consist of the linking numbers between components of $L$. Write $\Lambda_1,\dots, \Lambda_\ell$ for the columns of $L$. 

Write $\g=a_1\mu_1+\cdots +a_\ell \mu_\ell$, viewed as an element of $\Z^\ell$.  Note that $\g\in \Z^\ell$ becomes 0 in $H_1(S^3_\Lambda(L))/\Tors$ if and only if $N\cdot \g\in \im \Lambda$ for some $N\in \Z$.
 We will construct a null-homotopy of the map $\frA_{\g}$ from Equation~\eqref{eq:hypercube-homology-action} whenever $\g=0\in H_1(S^3_{\Lambda}(L))/\Tors$.

We first construct a function $\eta_\g\colon \bH(L)\to \bZ$ satisfying
\begin{equation}
\eta_\g(\ve{s}+\Lambda_i)=\eta_\g(\ve{s})-a_i. \label{eq:coherence-eta}
\end{equation}
To construct such an $\eta_\g$, we pick representatives in $\bH(L)$ of each class in $\bH(L)/\im \Lambda$ and define $\eta_\g$ arbitrarily on these elements. We extend $\eta_\g$ to all of $\bH(L)$ using Equation~\eqref{eq:coherence-eta}. To see that the resulting map $\eta_\g$ is well-defined, it suffices to show that if 
\begin{equation}
j_1 \Lambda_1+\cdots+j_\ell \Lambda_\ell=0,\label{eq:sum-j-Lambda}
\end{equation}
 then
 \[
 j_1 a_1+\cdots+ j_\ell a_\ell=0.
 \] 
To see this, observe that Equation~\eqref{eq:sum-j-Lambda} implies that $(j_1,\dots, j_\ell)$ is in the null-space of $\Lambda$, so in particular its dot product with $(a_1,\dots,a_\ell)$ will vanish, since $(a_1,\dots, a_\ell)^T$ is in the rational image of $\Lambda$. Therefore an $\eta_\g$ satisfying Equation~\eqref{eq:coherence-eta} exists.

We now define the function
$\omega_\g\colon \bH(L)\times \bE_\ell\to \Z$
by the formula
\[
\omega_\g(\ve{s}, \veps)=\eta_\g(\ve{s})+\veps_1\cdot a_1+\cdots +\veps_\ell \cdot a_\ell.
\]
We observe that $\omega_\g$ satisfies the formula 
\begin{equation}
\omega_\g(\ve{s}, \veps+e_i)=a_{i}+\omega_\g(\ve{s},\veps)\quad \text{and} \quad \omega_\g(\ve{s}+\Lambda_{i},\veps+e_i)=\omega_\g(\ve{s},\veps). 
\label{eq:requirements-omega}
\end{equation}

We now define a homotopy $H\colon \cC_{\Lambda}(L)\to \cC_{\Lambda}(L)$ as follows. If $\ve{s}\in \bH(L)$ and $\veps\in \bE_\ell$, write $\cC_{\veps}(\ve{s})\subset \cC_{\veps}\subset \cC_{\Lambda}(L)$ for the subspace in Alexander grading $\ve{s}$. We define $H$ on $\cC_{\veps}(\ve{s})$ via the formula
\[
H(\xs)=\omega_\g(\ve{s},\veps)\cdot \xs.
\]
We observe that
 \[
 \frA_{\g}=[D, H],
 \]
 where $D$ is the hypercube differential on $\cC_{\Lambda}(L)$. The proof is complete.
\end{proof}

\begin{rem} It is also straightforward to verify that the above action descends to an action of $\Lambda^* \left(H_1(S^3_{\Lambda}(L))/\Tors\right)$. This may be verified by noting that for each $i$ we have $\frA_{\mu_i}^2=0$, and for all $i,j$ we have $[\frA_{\mu_i},\frA_{\mu_j}]=0$.
\end{rem}

\begin{rem}
\label{rem:subcube-refinement}
It is helpful to have the following refinement of Lemma~\ref{lem:algebraic-homology-action-link-surgery} for subcubes of $\cC_{\Lambda}(L)$. Suppose that $L$ is partitioned as $L_0\cup L_1$ and let $\veps\in \bE_{\ell_0}$ be a fixed coordinate, where $\ell_i=|L_i|$. Consider the subcube $\cC_{\Lambda_0}^{(*,\veps)}(L_0;L_1)\subset \cC_{\Lambda}(L)$ generated by complexes at points $(\nu,\veps)\in \bE_{\ell_0}\times \bE_{\ell_1}$ where $\veps$ is our chosen coordinate (above) and $\nu$ is any coordinate. Here $\Lambda_0$ is the restriction of $\Lambda$ to $L_0$. The complex $\cC_{\Lambda_0}^{(*,\veps)}(L_0;L_1)$ is a module over $\bF\llsquare \scU_{\ell_0+1},\scV_{\ell_0+1},\dots, \scU_{\ell_0+\ell_1},\scV_{\ell_0+\ell_1}\rrsquare $ (the variables for $L_1$). Furthermore, by \cite{ZemBordered}*{Lemma~7.7}, the differential on $\cC_{\Lambda_0}^{(*,\veps)}(L_0;L_1)$ commutes with the action of $\bF\llsquare \scU_{\ell_0+1},\scV_{\ell_0+1},\dots, \scU_{\ell_0+\ell_1},\scV_{\ell_0+\ell_1}\rrsquare $.  If $\g\in \Z^{\ell_0}$, then we may define an endomorphism $\frA_{\g}$ on $\cC_{\Lambda_0}^{(*,\veps)}(L_0;L_1)$ using  Equation~\eqref{eq:hypercube-homology-action}. Similarly to Lemma~\ref{lem:algebraic-homology-action-link-surgery}, the map $\frA_{\g}$ gives an action of $H_1(S^3_{\Lambda_0}(L_0))/\Tors$ which is well-defined up to $\bF\llsquare \scU_{\ell_0+1},\scV_{\ell_0+1},\dots, \scU_{\ell_0+\ell_1},\scV_{\ell_0+\ell_1}\rrsquare $-equivariant chain homotopy. To ensure the $\bF\llsquare \scU_{\ell_0+1},\scV_{\ell_0+1},\dots, \scU_{\ell_0+\ell_1},\scV_{\ell_0+\ell_1}\rrsquare $-equivariance of the homotopy, we add to Equation~\eqref{eq:requirements-omega} the requirement that if $e_i\in \Z^{\ell_0+\ell_1}$ is a unit vector pointing in the direction for $K_i\subset L_1$, then 
\[
\omega_\g(\ve{s}\pm e_i,\veps)=\omega_\g(\ve{s}, \veps).
\]
Since $\scU_i$ and $\scV_i$ have Alexander grading $\pm e_i$, this ensures that our chain homotopies will be $\bF\llsquare \scU_{\ell_0+1},\scV_{\ell_0+1},\dots, \scU_{\ell_0+\ell_1},\scV_{\ell_0+\ell_1}\rrsquare $-equivariant. Noting that there is an affine isomorphism $\bH(L)/\Z^{\ell_1}\iso \bH(L_0)$ (where $\Z^{\ell_1}$ acts by the meridians of the components of $L_1$), the proof in the absolute case goes through without change.
\end{rem}

Finally, we remark that in the subsequent work \cite{ZemGenSurgery}*{Section~11}, we prove that the  action $\frA_{\g}$ coincides with the $\Lambda^* H_1(S^3_{\Lambda}(L))/\Tors$-action on $\ve{\CF}^-(S^3_{\Lambda}(L))$. We do not use that fact in this paper.

\subsection{A second homology action}
\label{sec:diagrammatic-action}
In this section, we define another homology action on the link surgery formula.  This construction is defined in  \cite{ZemBordered}*{Section~13.3} (cf. \cite{HHSZNaturality}*{Section~6.2}). This version of the homology action will take the form of an endomorphism
\[
\cA_{\g}\colon \cC_{\Lambda}(L)\to \cC_{\Lambda}(L)
\]
associated to an immersed 1-chain $\g$ on the Heegaard surface $\Sigma$. This description parallels the construction of the ordinary $\Lambda^* H_1(Y)/\Tors$ action on Heegaard Floer homology \cite{OSDisks}*{Section~4.2.5}. 

 Suppose $L\subset S^3$ is a framed link with a system of arcs $\scA$.  Let $\scH$ be a $\sigma$-basic system of Heegaard diagrams for $(L,\scA)$. Let $\Sigma$ be the underlying Heegaard surface of $\scH$, and let us also assume that the arcs $\scA$ are embedded in $\Sigma$. Let $\g\subset \Sigma$ be a closed curve which is disjoint from the arcs of $\scA$. We define an endomorphism
\[
\cA_{\g}\colon \cC_{\Lambda}(L,\scA)\to \cC_{\Lambda}(L,\scA).
\]
The endomorphism $\cA_\g$ is induced by a type-$D$ endomorphism  $\cA_{\g}^1$ of $\cX_{\Lambda}(L,\scA)^{\cL}$.

The construction of $\cA_\g$ is as follows.   The hypercube $\cC_{\Lambda}(L,\scA)$ is built from an $|L|$-dimensional hyperbox of chain complexes. Each constituent hypercube is obtained by pairing two hypercubes of attaching curves, of combined total dimension at most $|L|$, and then extending the remaining axis directions by canonical diffeomorphism maps for surface isotopes of $\Sigma$ which push basepoints along subarcs of the curves in $\scA$.

Suppose a subcube of this hyperbox is formed by pairing hypercubes of attaching curves $\cL_{\a}$ and $\cL_{\b}$, of dimension $n$ and $m$, respectively. The homology action
\[
\cA_{\g}\colon \ve{\CF}^-(\cL_{\a},\cL_{\b})\to \ve{\CF}^-(\cL_{\a},\cL_{\b})
\]
is defined similarly to the hypercube differential in Equation~\eqref{eq:hypercube-differential}, except that a holomorphic polygon representing a class $\psi$ is weighted by a factor of
\[
\# (\d_{\a}(\psi)\cap \g):= \sum_{\a_\veps\in \cL_{\a}} \#(\d_{\a_\veps}(\psi)\cap \g).
\]
Here, if $\as_\veps\in \cL_{\a}$, then $\d_{\a_\veps}(\phi)$ denotes the subset of the boundary of the domain of $\phi$ which lies on $\as_\veps$. The homology action $\cA_{\g}$ on the link surgery hypercube $\cC_{\Lambda}(L,\scA)$ is obtained by performing the above construction to each constituent hypercube of the link surgery complex, and modifying weights of the variables similarly to the construction in Section~\ref{sec:basic-systems}.

\begin{rem}
\label{rem:g-intersects-scA}
 When $\g$ is not disjoint from $\scA$, it is slightly harder to define the map $\cA_\g$ because the basepoint moving maps associated to $\scA$ will change the curve $\g$ to some curve $\g'$. To define $\cA_\g$ on the entire hypercube in this case, we would need to pick a canonical homotopy between homology actions of $\cA_\g$ and $\cA_{\g'}$ on subcubes. Choosing a canonical homotopy turns out to be a subtle question that we will not address in this paper.
\end{rem}

\subsection{Relating the actions}
\label{sec:relating-actions}

In this section, we relate the two actions $\cA_{\g}$ and $\frA_{\g}$ constructed in the previous sections. Our main result is the following:

\begin{prop}
\label{prop:homology-action-description}
 Consider a meridional $\sigma$-basic system of Heegaard diagrams for a link $L\subset S^3$ and a system of arcs $\scA$ for $L$, where each arc is either alpha-parallel or beta-parallel. Let $\Sigma$ be the underlying Heegaard surface, and let $\mu_1,\dots, \mu_\ell \subset \Sigma$ be translates of the special alpha and beta meridians.
\begin{enumerate}[label=($\cA$-\arabic*),ref=$\cA$-\arabic*]
\item \label{prop:homology-action-claim-1} Suppose that $\g\subset \Sigma$ is a closed curve which is disjoint from the arcs of $\scA$. The endomorphism $\cA_\g$ of $\cC_{\Lambda}(L,\scA)$ satisfies
\[
\cA_\g\simeq \sum_{i=1}^\ell \lk(\g, K_i)\cdot \cA_{\mu_{i}}
\]
where each $\mu_{i}$ is parallel to the canonical meridian of the component $K_i$ on the diagram $\scH$. The linking numbers are computed by viewing each $K_i$ as being transverse to $\Sigma$ (intersecting positively at $z_i$ and negatively at $w_i$) while viewing $\g$ as being immersed on $\Sigma\subset S^3$. The same claim also holds when we view $\cA_\g$ and $\cA_{\mu_i}$ as type-$D$ endomorphisms of $\cX_{\Lambda}(L,\scA)^{\cL}$.
\item\label{prop:homology-action-claim-2} For each component $K_i$ of $L$, there is a chain homotopy
\[
\cA_{\mu_{i}}\simeq \frA_{\mu_{i}}
\]
as endomorphisms of $\cC_\Lambda(L,\scA)$. The same claim holds when we view $\cA_{\mu_i}$ and $\frA_{\mu_i}$ as type-$D$ endomorphisms of $\cX_{\Lambda}(L,\scA)^{\cL}$.
\end{enumerate}
\end{prop}

\begin{rem} The statement is more subtle when $\g$ intersects the arcs $\scA$. For example, if we let $\g$ be a very small loop centered at $w_i$, then $\lk(K_i,\g)=\pm 1$ and $\lk(K_j,\g)=0$ for $i\neq j$. The above formula predicts $\cA_\g\simeq \frA_{\mu_i}$. Note however that our definition of the map $\cA_{\g}$ is not complete in this case, since $\g$ intersects an arc of $\scA$. Cf. Remark~\ref{rem:g-intersects-scA}. 
\end{rem}

We begin with a preliminary lemma:

\begin{lem}[\cite{ZemBordered}*{Lemma~13.16}]
\label{lem:homology-action-homotopy-hypercube}
Suppose that $\cL_{\a}$ and $\cL_{\b}$ are hypercubes of attaching curves on $(\Sigma,\ws,\zs)$ and that $\g$ is a closed 1-chain on $\Sigma$. Suppose that $C\subset \Sigma$ is an integral 2-chain such that $\d C=\g+S_{\a}+S_{\b}$ where $S_{\a}$ are closed 1-chains which are disjoint from all curves in $\cL_{\a}$ and $S_{\b}$ are closed 1-chains which are disjoint from all curves in $\cL_{\b}$. As an endomorphism of $\ve{\CF}^-(\cL_{\a},\cL_{\b})$ we have
\[
\cA_\g\simeq 0.
\]
\end{lem}
\begin{proof} The proof is given in \cite{ZemBordered}*{Lemma~13.16}, though we repeat it for the benefit of the reader. We construct the following diagram to realize the chain homotopy:
\begin{equation}
\begin{tikzcd} \ve{\CF}^-(\cL_{\a},\cL_{\b})
	\ar[r, "\cA_{\g}"]
	\ar[d, "\id"]
	\ar[dr, "H_C",dashed]
& \ve{\CF}^-(\cL_{\a},\cL_{\b})
	\ar[d, "\id"]
\\
\ve{\CF}^-(\cL_{\a},\cL_{\b})
	\ar[r, "0"]&
\ve{\CF}^-(\cL_{\a},\cL_{\b})
\end{tikzcd}
\label{eq:hypercube-A_gamma-null-homotopy}
\end{equation}
We define the map $H_C$ to have only length 2 chains in the above diagram, and to send $\xs\in \bT_{\a_\veps}\cap \bT_{\b_{\nu}}$ to  $n_{\xs}(C)\cdot \xs$, and to be $\bF\llsquare U_1,\dots, U_\ell\rrsquare $-equivariant. Here $n_{\xs}(C)\in \bF$ denotes the intersection number of $C\subset \Sigma$ with $\xs$, where we view $\xs$ as a collection of points on $\Sigma$. 

We claim that Equation~\eqref{eq:hypercube-A_gamma-null-homotopy} is a hypercube. It suffices to show the following equation whenever $(\nu,\veps)\le (\nu',\veps')$:
\begin{equation}
H_C\circ D_{(\nu,\veps),(\nu',\veps')}+ D_{(\nu,\veps),(\nu',\veps')}\circ H_C=(\cA_\g)_{(\nu,\veps),(\nu',\veps')}
\label{eq:homology-action-homotopy-hypercube}
\end{equation}
The maps $D_{(\nu,\veps),(\nu',\veps')}$ and $(\cA_\g)_{(\nu,\veps),(\nu',\veps')}$ count the same holomorphic polygons, but with different weights. To prove Equation~\eqref{eq:homology-action-homotopy-hypercube}, we observe that if 
\[
\psi\in \pi_2(\Theta_{\nu_n,\nu_{n-1}},\dots, \Theta_{\nu_2,\nu_1} ,\xs,\Theta_{\veps_1,\veps_2},\dots ,\Theta_{\veps_{m-1},\veps_m},\ys)
\]
 is a class of $(n+m)$-gons, then
\[
0\equiv\#\d(\d_{\a}(\psi)\cap C)\equiv n_{\ve{x}}(C)+n_{\ve{y}}(C)+\#\d_{\a}(\psi)\cap (\g+S_{\a}+S_{\b}) \pmod 2,
\]
We note that $\d_{\a}(\psi)\cap S_{\a}=\d_{\b}(\psi)\cap S_{\b}=\emptyset$. Furthermore $\d_{\a}(\psi)$ and $\d_{\b}(\psi)$ are homologous via the domain of the class $\psi$ (an integral 2-chain $D(\psi)$ on $\Sigma$), so we also have $\# (\d_{\a}(\psi)\cap S_{\b})\equiv 0$. Hence
\[
n_{\ve{x}}(C)+n_{\ve{y}}(C)\equiv \#(\d_{\a}(\psi)\cap \g)\pmod 2.
\]
This implies Equation~\eqref{eq:homology-action-homotopy-hypercube}, completing the proof.
\end{proof}

We can extend the above result to the link surgery formula:

\begin{cor}
\label{cor:homology-action-equivariant} Suppose we pick a $\sigma$-basic system of Heegaard diagrams for $(L,\scA)$ with underlying Heegaard surface $\Sigma$, and $\g\subset \Sigma$ is a closed curve. Suppose that there is an integral 2-chain $C$ on $\Sigma$ such that $\g=\d C+S_{\a}+S_{\b}$, where $S_{\a}$ are closed 1-chains on $\Sigma$ which are disjoint from all alpha curves, and $S_{\b}$ are closed 1-chains disjoint from all beta curves. Suppose furthermore that $S_{\a}$ and $S_{\b}$ are disjoint from the arcs $\scA$. Then $\cA_{\g}\simeq 0$ as an endomorphism of both $\cC_{\Lambda}(L,\scA)$ and $\cX_{\Lambda}(L,\scA)^{\cL}$.
\end{cor}
\begin{proof} 
The claim that $\cA_{\g}\simeq 0$ as an endomorphism of $\cC_{\Lambda}(L)$ follows by applying Lemma~\ref{lem:homology-action-homotopy-hypercube} to each constituent hypercube of the hyperbox used in the construction of the link surgery formula.

We now consider the claim at the level of type-$D$ morphisms. In this case, the statement follows from the fact that the homotopy we constructed is suitably equivariant, as we now describe. Suppose $ \cX^{\cK}$ and $\cY^{\cK}$ are type-$D$ modules and
\[
F\colon \cX^{\cK}\hatbox {}_{\cK} \cD_0\to \cY^{\cK}\hatbox {}_{\cK} \cD_0
\]
is a linear map (not necessarily a chain map). We suppose further that we may write $F$ as in the following diagram:
\[
\begin{tikzcd}[row sep=1.3cm, labels=description]
 \cX^{\cK}\hatbox {}_{\cK} \cD_0\ar[d, "F"]
\\
\cY^{\cK}\hatbox {}_{\cK} \cD_0
\end{tikzcd}=
\quad
\begin{tikzcd}[row sep=1.5cm, column sep=1.5cm, labels=description]
X_0 \ar[r, "v+h"] \ar[d, "F_0"] \ar[dr, "F_v+F_h"] & X_1 \ar[d, "F_1"]\\
Y_0 \ar[r, "v+h"] & Y_1,
\end{tikzcd}
\]
where $v$ and $h$ are the chain maps induced by the terms of $\delta^1$ weighted by $\sigma$ and $\tau$, respectively. Finally, we observe that $F=f^1\boxtimes \bI$ for a type-$D$ module map $f^1\colon \cX^{\cK}\to \cY^{\cK}$ if
\begin{enumerate}[label= ($e$-\arabic*), ref=$e$-\arabic*]
\item \label{equivariance-1} $F_0$ is $\bF[\scU,\scV]$-equivariant;
\item\label{equivariance-2} $F_1$ is $\bF[\scU,\scV,\scV^{-1}]$-equivariant;
\item\label{equivariance-3} $F_v (a\cdot x)=\phi^\sigma(a)\cdot F_v(x)$ and $F_h(a\cdot x)=\phi^\tau(a)\cdot F_h(x)$, where $\phi^\sigma$ and $\phi^\tau$ are the maps from Section~\eqref{sec:surgery-algebra}.
\end{enumerate}
The above statement generalizes easily to type-$D$ modules over $\cL$. In the case of the chain homotopy
\[
[\d, H]=\cA_{\g},
\]
constructed in Lemma~\ref{lem:homology-action-homotopy-hypercube}, we observe that $H$ satisfies equivariance properties analogous to ~\eqref{equivariance-1}--\eqref{equivariance-3} (but for more link components) since it just multiplies an intersection point $\xs\in \bT_{\a}\cap \bT_{\b}$ by $n_C(\xs)$, and is by definition extended equivariantly over the actions of the various variables. 
\end{proof}

We now prove the main result of this section: 

\begin{proof}[Proof of Proposition~\ref{prop:homology-action-description}]
We begin with the first claim. We use a meridional $\sigma$-basic system, as described in Section~\ref{sec:basic-systems}. Let $(\Sigma,\as,\bs,\ws,\zs)$ denote the underlying Heegaard link diagram for $(S^3,L)$. Write $(\Sigma,\as_0,\bs_0)$ for the partial Heegaard diagram obtained by deleting the special alpha and beta curves (i.e. the meridians of the components of $L$). If we attach compressing disks along the $\as_0$ and $\bs_0$ curves, we obtain $S^3\setminus \nu (L)$.  We write $\mu_1,\dots, \mu_\ell\subset \Sigma$ for curves which are parallel to the special meridional alpha and beta curves. The meridians $\mu_1,\dots, \mu_\ell$  generate $H_1(S^3\setminus \nu(L))$. Hence, $\g$ is homologous in $S^3\setminus \nu(L)$ to a linear combination of these meridians, i.e. we may write
\[
\g-a_1\cdot \mu_1-\cdots -a_\ell \cdot\mu_\ell=\d C,
\]
for some 2-chain $C$ in $S^3\setminus \nu(L)$ and some integers $a_1,\dots, a_\ell$. Furthermore, we clearly have $a_i=\lk(\g, K_i)$. Such a relation induces a 2-chain $C'$ on $(\Sigma,\as_0,\bs_0)$ such that $\d C'$ is $\g-a_{1}\cdot \mu_1-\cdots -a_\ell\cdot \mu_\ell+ S_{\a}+ S_{\b}$, where $S_{\a}$ are parallel copies of curves in $\as_0$ and $\bs_0$, and $\mu_i$ are the canonical meridians. In particular, $S_{\a}$ is disjoint from all alpha curves in the $\sigma$-basic system and $S_{\b}$ is disjoint from all beta curves in the $\sigma$-basic system.  In our $\sigma$-basic system, the curves $\as_0$ and $\bs_0$ appear via small Hamiltonian translates in every collection of alpha or beta curves. Claim~\eqref{prop:homology-action-claim-1} about the map $\cA_\g$ as an endomorphism of $\cC_{\Lambda}(L,\scA)$ of the theorem statement now follows from Corollary~\ref{cor:homology-action-equivariant}.

We now prove claim~\eqref{prop:homology-action-claim-2} by constructing a homotopy $\cA_{\mu_i}\simeq F^{-K_i}$. We observe that $F^{-K_i}\simeq F^{K_i}$, which is by definition the same as $\frA_{\mu_i}$.  Write $\cL_{\a}$ and $\cL_{\b}$ for the hyperboxes of attaching curves from our $\sigma$-basic system. We suppose $K_i\subset L$ is a link component which is beta-parallel. We decompose the $K_i$-direction of $\cL_{\a}$ as
\[
\begin{tikzcd}
\cL_{\a_1}
\ar[r]& \cL_{\a_2}\ar[r] &\cdots \ar[r]& \cL_{\a_n}.
\end{tikzcd}
\]

As described in Section~\ref{sec:basic-systems}, the hypercube $\cC_{\Lambda}(L)$ is determined by the compression of the hypercube
\[
\begin{tikzcd}
\ve{\CF}^-(\cL_{\a_1},\cL_{\b})\ar[r, "F_{1,2}"]& \ve{\CF}^-(\cL_{\a_2},\cL_\b) \ar[r, "F_{2,3}"]&\cdots\ar[r, "F_{n-1,n}"]& \ve{\CF}^-(\cL_{\a_n},\cL_{\b}).
\end{tikzcd}
\]
For convenience, write $\cC_i$ for the hyperbox $\ve{\CF}^-(\cL_{\a_i},\cL_{\b})$.

Let $\mu_i'$ be a translation on the Heegaard diagram of $\mu_i$.
We observe that if $C\subset \Sigma$ is a 2-chain whose boundary is $\mu_i-\mu_i'$, then the following hyperbox compresses horizontally to the identity map on $\Cone(\cA_{\mu_i})$:
\begin{equation}
\begin{tikzcd}[labels=description, row sep=1cm, column sep=1cm]
\cC_n
	\ar[dr, dashed, "H_C"]
	\ar[r, "\id"]
	\ar[d, "\cA_{\mu_i}"]
&
\cC_n
	\ar[d, "\cA_{\mu_i'}"]
	\ar[dr,dashed, "H_C"]
	\ar[r, "\id"]
&
\cC_n
	\ar[d, "\cA_{\mu_i}"]
\\
\cC_n
	\ar[r, "\id"]
&
\cC_n
	\ar[r, "\id"]
&
\cC_n
\end{tikzcd}
\label{eq:move-move-back}
\end{equation}

Next, we observe that for any $k\in \{1,\dots, n-1\}$, the horizontal compressions of the following hyperboxes are homotopic as hyperbox morphisms from $\Cone(\cA_{\mu_i}\colon \cC_k\to \cC_k)$ and $\Cone(\cA_{\mu_i'}\colon \cC_{k+1}\to \cC_{k+1})$:
\begin{equation}
\begin{tikzcd}[labels=description,column sep=1.4cm, row sep=1.2cm]
\cC_k
	\ar[dr, dashed, "F_{k,k+1}^{\mu_i}"]
	\ar[r, "F_{k,k+1}"]
	\ar[d, "\cA_{\mu_i}"]
&
\cC_{k+1}
	\ar[d, "\cA_{\mu_i}"]
	\ar[dr,dashed, "H_C"]
	\ar[r, "\id"]
&
\cC_{k+1}
	\ar[d, "\cA_{\mu_i'}"]
\\
\cC_k
	\ar[r, "F_{k,k+1}"]
&
\cC_{k+1}
	\ar[r, "\id"]
&
\cC_{k+1}
\end{tikzcd}
\quad \text{and} \quad
\begin{tikzcd}[labels=description, column sep=1.4cm, row sep=1.2cm]
\cC_k
	\ar[dr, dashed, "H_C"]
	\ar[r, "\id"]
	\ar[d, "\cA_{\mu_i}"]
&
\cC_{k+1}
	\ar[d, "\cA_{\mu_i'}"]
	\ar[dr,dashed, "F_{k,k+1}^{\mu_i'}"]
	\ar[r, "F_{k,k+1}"]
&
\cC_{k+1}
	\ar[d, "\cA_{\mu_i'}"]
\\
\cC_k
	\ar[r, "\id"]
&
\cC_{k+1}
	\ar[r, "F_{k,k+1}"]
&
\cC_{k+1}
\end{tikzcd}
\label{eq:two-compressions}
\end{equation}
The above claim follows from Lemma~\ref{lem:homology-action-homotopy-hypercube}, which constructs a hyperbox which realizes a chain homotopy $\cA_{\mu_i}\simeq \cA_{\mu_i'}$ as endomorphisms of $\Cone(F_{j,j+1})$. This hyperbox realizes exactly the chain homotopy between the compressions described above.

In particular, from the two results related to Equations~\eqref{eq:move-move-back} and ~\eqref{eq:two-compressions}, we conclude that the hypercube map $\cA_{\mu_i}\colon \cC_{\Lambda}(L)\to \cC_{\Lambda}(L)$ can be described as the compression of the following hyperbox, for any $k$:
\[
\begin{tikzcd}[column sep=.6cm, row sep=1cm]
\cC_1
	\ar[r, "F_{1,2}"]
	\ar[dr,dashed]
	\ar[d, "\cA_{\mu_i}",labels=description]
& 
\cC_2
	\ar[r, "F_{2,3}"]
	\ar[d, "\cA_{\mu_i}",labels=description]
	\ar[dr,dashed]
&
\cdots
\ar[r]
&
\cC_k
	\ar[r, "\id"]
	\ar[d, "\cA_{\mu_i}",labels=description]
	\ar[dr,dashed, "H_C",labels=description]
& 
\cC_k
	\ar[r,"F_{k,k+1}"]
	\ar[d, "\cA_{\mu_i'}",labels=description]
	\ar[dr,dashed]
&
\cC_{k+1}
	\ar[r]
	\ar[d, "\cA_{\mu_i'}",labels=description]
&
\cdots 
\ar[r]
&
\cC_n
	\ar[r, "\id"]
	\ar[d, "\cA_{\mu_i'}",labels=description]
	\ar[dr, dashed,"H_C",labels=description]
&\cC_n
	\ar[d, "\cA_{\mu_i}",labels=description]
\\
\cC_1\ar[r, "F_{1,2}"]& \cC_2\ar[r, "F_{2,3}"]& \cdots\ar[r]& \cC_k\ar[r, "\id"]& \cC_k\ar[r,"F_{k,k+1}"]&\cC_{k+1}\ar[r]& \cdots \ar[r] &\cC_n\ar[r, "\id"]& \cC_n
\end{tikzcd}
\]
We may pick $C$, $\mu_i$, $\mu_i'$ and $k$ so that $\mu_i$ is disjoint from $\cL_{\a_1},\dots, \cL_{\a_{k}}$, and $\mu_i'$ is disjoint from $\cL_{\a_{k+1}},\dots, \cL_{\a_n}$, and $C$ covers all of the special meridional curve of $\cL_{\a_k}$ (and in particular, $n_C(p)\equiv 1$ for each $p$ in the special meridional curve of any curve collection in $\cL_{\a_k}$). We assume, additionally, that $C$ is disjoint from the curves of $\cL_{\a_n}$ (in particular, the initial special meridional curve $\a_i^s$). See Figure~\ref{cor:1}. In particular, the only diagonal map which will be non-trivial in the above compression will be the one with domain $\cC_k$. Since $n_C(\xs)\equiv 1$ for every intersection point $\xs$ in any of the Floer complexes comprising $\cC_k$, the map $H_C$ will act by the identity on $\cC_k$.  In particular, the compression coincides with the diagram
\[
\begin{tikzcd}\cC_{\Lambda}(L)\ar[d, "\cA_{\mu_i}"] \\ \cC_{\Lambda}(L)
\end{tikzcd}
\simeq
\begin{tikzcd}[labels=description, column sep=2cm]
\cC^{(*,0)}(L)\ar[d, "0"]\ar[r, "F^{-K_i}"]\ar[dr,dashed, "F^{-K_i}"]& \cC^{(*,1)}(L)\ar[d, "0"]\\
\cC^{(*,0)}(L)\ar[r, "F^{-K_i}"]& \cC^{(*,1)}(L)
\end{tikzcd}.
\]
The proof of ~\eqref{prop:homology-action-claim-2} is complete, when viewing the maps as endomorphisms of $\cC_{\Lambda}(L,\scA)$. The argument for the endomorphisms at the level of type-$D$ morphisms follows from the same reasoning as the analogous claim in ~\eqref{prop:homology-action-claim-1}, so the proof is complete.
\end{proof}

\begin{figure}[ht]
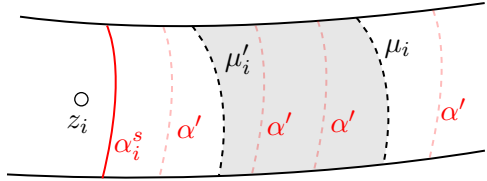
\caption{The curves $\mu_i$, $\mu_i'$ and the special meridian $\a_i^s$ of component $J\subset L$. The shaded annulus is the 2-chain $C$. The dashed curves labeled $\a'$ are the translates of $\a_i^s$ appearing in the $\sigma$-basic system.}
\label{cor:1}
\end{figure}

\begin{example} We illustrate Proposition~\ref{prop:homology-action-description} for the knot surgery formula. In this case, the result is already known (using different notation). See \cite{OSProperties}*{Theorem~9.23}. To translate into our present notation, recall that Ozsv\'{a}th and Szab\'{o}'s mapping cone formula \cite{OSIntegerSurgeries} takes the form
\[
\ve{\CF}^-(S^3_n(K))\iso \Cone(v+h_n\colon \bA(K)\to \bB(K)).
\]
 Proposition~\ref{prop:homology-action-description} translates to the claim that the diagrammatic homology action for the meridian $\mu$ coincides with the map $h_n$, viewed as an endomorphism of the mapping cone complex. (Of course, it is straightforward to see that this is null-homotopic when $n\neq 0$).  
\end{example}

\subsection{Application to the connected sum formula}
\label{sec:simplifying-connected-sum-1}

An important application of the results of the previous section is the following:

\begin{cor}
\label{cor:simplify-tensor-product}
 Suppose that $L_1$ and $L_2$ are links in $S^3$ with framings $\Lambda_1$ and $\Lambda_2$, and with systems of arcs $\scA_1$ and $\scA_2$, respectively. Let $K_1\subset L_1$ and $K_2\subset L_2$ be distinguished components, whose arcs are alpha-parallel and beta-parallel, respectively. Let $\Lambda_1'$ denote the restriction of the framing $\Lambda_1$ to $L_1\setminus K_1$. Suppose that $K_1\subset S^3_{\Lambda_1'}(L_1\setminus K_1)$ is rationally null-homologous. 
Write $\cX_{\Lambda_i}(L_i; K_i)^{\cK}$ for the type-$D$ modules obtained from $\cX_{\Lambda_i}(L_i,\scA_i)^{\cL_{\ell_i}}$ by boxing ${}_{\cK} \cD_0$ into each algebra component except for the one corresponding to $K_i$.
Then
\[
\begin{split}
&\left(\cX_{\Lambda_1}(L_1; K_1)^{\cK}, \cX_{\Lambda_2}(L_2;K_2)^{\cK}\right)\hatbox {}_{\cK| \cK} W_l^{\cK}
\\
\simeq  &\left(\cX_{\Lambda_1}(L_1;K_1)^{\cK}, \cX_{\Lambda_2}(L_2;K_2)^{\cK}\right)\hatbox {}_{\cK|\cK} M^{\cK}.
\end{split}
\]
\end{cor}

\begin{proof} We use the description of the surgery complex from Theorem~\ref{thm:pair-of-pants-on-hypercube}.  In the model constructed using ${}_{\cK|\cK}W_l^{\cK}$, we have the term
\[
\Tw_{K_1}\circ (F^{-K_1}|F^{-K_2})
\]
where
\[
\Tw_{K_1}=\id+  \scV^{-1}(\Phi_{w_1}|1+1|\Phi_{w_2})\circ (\cA_{K_1} | 1).
\]
In the model constructed using ${}_{\cK_1|\cK_2} M^{\cK}$, we instead have the term
\[
F^{-K_1}|F^{-K_2}.
\]
The map $\cA_{K_1}$ appears as an endomorphism of $\cC_{\Lambda_1}^{(*,1)}(L_1)$. We claim that $\cA_{K_1}$ is null-homotopic (and furthermore, that the corresponding type-$D$ morphism is null-homotopic on idempotent 1 of the type-$D$ module $\cX_{\Lambda_1}(L_1,\scA_1)^{\cK}$). We now use Proposition~\ref{prop:homology-action-description} to construct a chain homotopy
\[
\cA_{K_1}\simeq \sum_{i=1}^n \lk(K_1,K_i)\cdot  \frA_{\mu_i}.
\]
Note that the term $\lk(K_1,K_1)$ will depend on how we embed a trace of the knot $K_1$ on the Heegaard diagram. It is the linking number of a trace of $K_1$ on $\Sigma$ with a copy of $K_1$, pushed to be transverse to $\Sigma$, as in the statement of Proposition~\ref{prop:homology-action-description}. Note that on $\cC^{(*,1)}(L_1)$, the term $\frA_{\mu_1}$ vanishes identically, so we obtain that $\cA_{K_1}$ is homotopic on $\cC^{(*,1)}(L_1)$ to
\[
\sum_{i=2}^n \lk(K_1,K_i)\cdot  \frA_{\mu_i}.
\]
We now observe that Lemma~\ref{lem:algebraic-homology-action-link-surgery} implies that the above is homotopic to $\frA_{K_1}$, where we now view $K_1$ as an element of $H_1(S^3_{\Lambda'_1}(L_1\setminus K_1))/\Tors$ (cf. Remark~\ref{rem:subcube-refinement}). Hence $\cA_{K_1}\simeq 0$ on $\cC^{(*,1)}(L_1)$. It follows that $\Tw_{K_1}\simeq \id$, so the claim follows on the level of surgery complexes. To see that the claim also holds on the level of type-$D$ modules, it suffices to show that the homotopy $\Tw_{K_1}\simeq \id$ is $\bF[\scU,\scV,\scV^{-1}]$-equivariant. However this follows from Remark~\ref{rem:subcube-refinement}. 
\end{proof}

\section{Simplifying the pairing theorem}

In this section, we recall the category of \emph{Alexander modules} from \cite{ZemBordered}*{Section~6}. After recalling precise definitions, we will prove the following result, which is essential to our proof of Theorem~\ref{thm:pairing-main} when $b_1(Y(G))>0$. The results of this section are not essential when $b_1(Y(G))=0$.

\begin{prop}\label{prop:merge=pair-of-pants}
 Suppose that $\cX_1^{\cK}$ and $\cX_2^{\cK}$ are type-$D$ Alexander modules which are homotopy equivalent to finitely generated type-$D$ modules (i.e. have homotopy equivalent models where $\cX_i$ are finite dimensional $\bF$-vector spaces). Then
 \[
\left( \cX_1^{\cK},\cX_2^{\cK}\right)\hatbox {}_{\cK| \cK} W_l^{\cK}\simeq \left( \cX_1^{\cK}, \cX_2^{\cK}\right)\hatbox {}_{\cK| \cK} M^{\cK}.
 \]
 The same holds for $W_r$ in place of $W_l$.
\end{prop}

Our proof is inspired by work of the author with Hendricks and Manolescu \cite{HMZConnectedSum}*{Section~6}, where an extra term in an involutive connected sum formula is shown to be null-homotopic for purely algebraic reasons. The proof occupies the next several subsections.

\subsection{Finitely generated $\cK$-modules}

 We say that $\cX^{\cK}$ is \emph{finitely generated} if $\cX$ is finitely generated as an $\bF$ vector space. Type-$D$ Alexander modules which are finitely generated form a subcategory which admits a somewhat simpler description not involving linear topological spaces.

We denote by $\ve{\cK}$ the completion of the algebra $\cK$ with respect to the topology described in Section~\ref{def:knot-topology}. As a vector space, we have the following isomorphisms:
\[
\begin{split}
\ve{I}_0\cdot \ve{\cK} \cdot \ve{I}_0&\iso \bF\llsquare \scU,\scV\rrsquare \\
 \ve{I}_1\cdot \ve{\cK}\cdot \ve{I}_0&\iso \bF[\scV,\scV^{-1}\rrsquare  \llsquare \scU\rrsquare  \langle \tau\rangle \oplus \bF\llsquare \scV,\scV^{-1}] \llsquare \scU\rrsquare  \langle \sigma \rangle\\
\ve{I}_1\cdot \ve{\cK} \cdot \ve{I}_1&\iso \bF[\scV,\scV^{-1}] \llsquare \scU\rrsquare .
\end{split}
\]

In the above, $\bF[\scV,\scV^{-1}\rrsquare$ denotes the set of half infinite Laurent series in $\scV$, completed with respect to negative powers. More explicitly, these are infinite sums $\sum_{i\in \Z} a_i \scV^{i}$ where $a_i\in \bF$ and there is an $N\in \Z$ so that $a_i=0$ if $i>N$. Also, $\bF[\scV,\scV^{-1}\rrsquare  \llsquare \scU\rrsquare$ consists of infinite sums
\[
\sum_{i\ge 0} f_i \scU^i
\]
where $f_i\in \bF[\scV,\scV^{-1}\rrsquare$. (Note that $\bF[\scV,\scV^{-1}\rrsquare  \llsquare \scU\rrsquare$ and $\bF\llsquare \scU\rrsquare [\scV, \scV^{-1}\rrsquare$ are not the same).

The completion $\ve{\cK}$ is an algebra, and $\mu_2$ is continuous when viewed as a map from $\ve{\cK} \vecotimes_{\mathrm{I}} \ve{\cK}$ to $\ve{\cK}$. 

Since $\ve{\cK}$ is an algebra, we may consider the category of finitely generated type-$D$ modules over $\ve{\cK}$. We write $\MOD^{\ve{\cK}}_{\fg}$ for this category. These are ordinary type-$D$ modules over $\ve{\cK}$ (i.e. not equipped with a topology), which have underlying vector spaces which are finitely generated over $\bF$. Morphisms in this category are ordinary type-$D$ morphisms, i.e. maps $f^1\colon \cX\to \cY\otimes \ve{\cK}$. 

 There is a related category, $\MOD^{\cK}_{\fg,\fra}$, which is the set of type-$D$ Alexander modules over $\cK$, which are finitely generated over $\bF$.

\begin{lem}[\cite{ZemBordered}*{Section~6.7}] The categories $\MOD^{\cK}_{\fg,\fra}$ and $\MOD^{\ve{\cK}}_{\fg}$ are equivalent.
\end{lem}

The category of finitely generated type-$D$ modules is important for the bordered theory for the following reason.

\begin{prop}[\cite{ZemBordered}*{Proposition~18.11}]
\label{prop:finite-generation}
Suppose $L=K_1\cup \cdots \cup K_n$ is a link in $S^3$ with framing $\Lambda$. Write $\cX_{\Lambda}(L; K_n)^{\cK}$ for the type-$D$ module obtained by tensoring $(n-1)$-copies of the type-$A$ module for a solid torus, ${}_{\cK} \cD_0$,  with $\cX_{\Lambda}(L)^{\cL_n}$ along the algebra components for $K_1,\dots, K_{n-1}$. Then $\cX_{\Lambda}(L; K_n)^{\cK}$ is homotopy equivalent to a finitely generated type-$D$ module.
\end{prop}

Note that an alternate proof may be found in \cite{ZemGenSurgery}*{Theorem~9.1}, which proves that $\cX_{\Lambda}(L;K_n)^{\cK}$ is homotopy equivalent to $\cX_{\lambda}(S_{\Lambda'}(L'), K_n)^{\cK}$, where $L'=L\setminus K$, and $\Lambda'$ and $\lambda$ are the restricted framing. This module is finitely generated since it is constructed using generators on a finite set of Heegaard diagrams.

For the sake of exposition, we now explain the definition of a finitely generated type-$D$ module over $\ve{\cK}$ in more detail:

\begin{rem}
\label{rem:restate-type-D}A finitely generated type-$D$ Alexander module over $\cK$ is equivalent to the following data:
\begin{enumerate}
\item A pair of finite dimensional $\bF$-vector spaces $\cX_0$ and $\cX_1$, equipped with internal differentials
\[
\d_0\colon \cX_0\to \cX_0\otimes \bF\llsquare \scU,\scV\rrsquare \quad \text{and} \quad \d_1\colon \cX_1\to \cX_1\otimes \bF[\scV,\scV^{-1}]\llsquare \scU\rrsquare .
\]
\item Maps
\[
\begin{split}
v&\colon \cX_0\to  \cX_1\otimes \bF\llsquare \scV,\scV^{-1}]\llsquare \scU\rrsquare \quad \text{and}
\\ h&\colon \cX_0\to \cX_1\otimes \bF[\scV,\scV^{-1}\rrsquare \llsquare \scU\rrsquare .
\end{split}
\]
 Furthermore, if we extend $v$ $\phi^\sigma$-equivariantly to a map $V$ with domain $\cX_0\otimes \bF\llsquare \scU,\scV\rrsquare $, then $V$ is a chain map. If we extend $h$ $\phi^\tau$-equivariantly to a map $H$ with domain $\cX_0\otimes \bF\llsquare \scU,\scV\rrsquare $, then $H$ is a chain map. (Recall that $\phi^\sigma$ and $\phi^\tau$ are the algebra morphisms in the definition of $\cK$; see Section~\ref{sec:surgery-algebra}).
\end{enumerate}
Morphisms between finitely generated type-$D$ modules are similar. As a particularly important special case, suppose $h$ and $h'$ are two maps from $\cX_0$ to $\cX_1\otimes \bF[\scV,\scV^{-1}\rrsquare \llsquare \scU\rrsquare $, whose $\phi^\tau$-equivariant extensions are chain maps, as above. Suppose there is a third map $j\colon \cX_0\to \cX_1\otimes \bF[\scV,\scV^{-1}\rrsquare \llsquare \scU\rrsquare $, and let $J$ be its $\phi^\tau$-equivariant extension to $\cX_0\otimes \bF\llsquare \scU,\scV\rrsquare $. If $\d_1\circ J+J\circ \d_0=H+H'$, then the type-$D$ module obtained by replacing $h$ with $h'$ is homotopy equivalent to $\cX^{\cK}$.
\end{rem}

We now prove a key lemma:

\begin{lem}\label{lem:classification-PID} Suppose that $\cX^{\cK}$ is a finitely generated type-$D$ module such that $\cX_1\otimes \bF[\scV,\scV^{-1}\rrsquare \llsquare \scU\rrsquare $ admits a $\Z_2$-valued grading (with $\scU$ and $\scV$ having even grading). Then the chain complex  $\cX_1\otimes \bF[\scV,\scV^{-1}\rrsquare \llsquare \scU\rrsquare $ is  chain isomorphic to a direct sum of 1-step complexes (i.e. complexes with a single generator and vanishing differential) and 2-step complexes (i.e. complexes with two generators $\xs,\ys$ and $\d(\xs)=\ys\otimes \a_{\ys,\xs}$ ). Furthermore, in the two step complexes, we may assume each $\a_{\ys,\xs}$ is of the form  $U^i$ for some $i\in \N$, where $U=\scU\scV$.
\end{lem}
\begin{proof} The key observation is that $\bF[\scV,\scV^{-1}\rrsquare $ is a field, so $\bF[\scV,\scV^{-1}\rrsquare \llsquare \scU\rrsquare $ is a PID. In fact, the ideals of $\bF[\scV,\scV^{-1}\rrsquare \llsquare \scU\rrsquare $ are all of the form $(U^i)$ for $i\ge 0$. Hence, the proof follows immediately from the classification theorem for finitely generated free chain complexes over a PID. We sketch the argument very briefly in our present setting. One first considers the arrows of the differential on $\cX_0\otimes \bF[\scV,\scV^{-1}\rrsquare \llsquare \scU\rrsquare $ which are weighted by units in $\bF[\scV,\scV^{-1}\rrsquare \llsquare \scU\rrsquare$. If any such arrow exists, we pick one arbitrarily. Suppose this arrow goes from $\xs$ to $\ys\otimes \a$. Since the ideals of $\bF[\scV,\scV^{-1}\rrsquare \llsquare \scU\rrsquare $ are all of the form $(U^i)$, one may perform a change of basis so that there are no other arrows from $\xs$ or to $\ys$. After performing this change of basis, there are also no arrows to $\xs$ or from $\ys$, since $\delta^1$ squares to 0. Performing a further change of basis yields the summand consisting of the subcomplex $\delta^1(\xs)= \ys\otimes 1$. We repeat this until all arrows (except on these two step complexes) have weight in the ideal $(U^1)$. We repeat the above procedure to isolate arrows with weight in $(U^1)\setminus (U^2)$ until, outside of the isolated 2-step complexes, all arrows have weight in the ideal $(U^2)$. We repeat this procedure until we are left with only 1-step and 2-step complexes. After a chain isomorphism, the weights on the algebra elements in the 2-step complexes may be taken to be powers of $U$.
\end{proof}

\begin{cor} \label{cor:eliminate-homology-action}
Let $\cX^{\cK}$ be a finitely generated type-$D$ module. Consider the endomorphism $\cA:=\scU \Phi+\scV\Psi$ of the chain complex $\cX_1\otimes \bF[\scV,\scV^{-1}\rrsquare \llsquare \scU\rrsquare $, where $\Phi$ and $\Psi$ denote the algebraically defined basepoint actions from Section~\ref{sec:background}. Then $\cA$ is null-homotopic on $\cX_1\otimes \bF[\scV,\scV^{-1}\rrsquare \llsquare \scU\rrsquare $ via an $\bF[\scV,\scV^{-1}\rrsquare \llsquare \scU\rrsquare $-equivariant chain homotopy. 
\end{cor}
\begin{proof} We recall that the endomorphisms $\Phi$ and $\Psi$ may be defined on any free chain complex over $\bF[\scV,\scV^{-1}\rrsquare \llsquare \scU\rrsquare$. The map $\Phi$ is the formal derivative of the differential with respect to $\scU$, and $\Psi$ is the formal derivative of the differential with respect to $\scV$. Moreover, if $C$ and $C'$ are two such free complexes and $F\colon C\to C'$ is an $\bF[\scV,\scV^{-1}\rrsquare \llsquare \scU\rrsquare$-equivariant chain map, then differentiating the equation
\[
\d \circ F+F\circ \d=0
\]
with respect to $\scU$ shows that $F$ homotopy commutes with $\Phi$. A similar argument works for $\Psi$. See \cite{ZemConnectedSums}*{Lemma~2.8} for arguments of a similar nature. 

Write $X$ for $\cX_1\otimes \bF[\scV,\scV^{-1}\rrsquare \llsquare \scU\rrsquare$. By the above observation, to show that $\scU\Phi +\scV\Psi$ is chain homotopic to zero on $X$, it suffices to find a free chain complex $C$ over $\bF[\scV,\scV^{-1}\rrsquare \llsquare \scU\rrsquare$ which is homotopy equivalent to $X$, and such that $\scU \Phi+\scV \Psi\equiv 0$ on $C$. We pick $C$ by performing a change of basis to $X$. Indeed using Lemma~\ref{lem:classification-PID}, we see that $X$ is chain isomorphic to a sum of 1-step complexes (i.e. a single generator $\xs$ with $\delta^1(\xs)=0$) and 2-step complexes spanned by generators $\xs$ and $\ys$ with differential of the form $\delta^1(\xs)=\ys\otimes U^i$. With respect to such a basis, we observe that $(\scU \d_{\scU}+\scV \d_{\scV})(U^i)=0,$ since $U=\scU\scV$.
\end{proof}

We are now in position to prove the main theorem of this section, Proposition~\ref{prop:merge=pair-of-pants}:

\begin{proof}[Proof of Proposition~\ref{prop:merge=pair-of-pants}] 
By Remark~\ref{rem:restate-type-D}, it is sufficient to show that the corresponding $h$ maps of the two type-$D$ modules are chain homotopic if we view the codomain (idempotent 1) as having coefficients in $\bF[\scV,\scV^{-1}\rrsquare \llsquare \scU \rrsquare$. We observe that by Equation~\eqref{eq:homology-action-basepoint} the difference between the two $h$ maps factors through the endomorphism $(\scU \Phi_1+\scV\Psi_1)\otimes \id$ of
\[
\big(\cX_{\Lambda_1}(L_1;K_1)\cdot \ve{I}_1\big)\otimes \big( \cX_{\Lambda_2}(L_2,K_2)\cdot \ve{I}_1\big)\otimes \bF[ \scV, \scV^{-1}\rrsquare \llsquare \scU\rrsquare.
\]
 Here, $\Phi_1$ and $\Psi_1$ denote the algebraic basepoint actions of $\cX_{\Lambda_1}(L_1;K_1)\cdot \ve{I}_1$.   Corollary~\ref{cor:eliminate-homology-action}, implies that the map $\scU \Phi_1+\scV \Psi_1$ is null-homotopic with these coefficients, so the proof is complete.
\end{proof}

\begin{rem} Our proof can also be used to show that  if $\cX^{\cK}$ is a type-$D$  Alexander  module which is homotopy equivalent to a finitely generated type-$D$ module, then there is a homotopy equivalence
\[
\cX^{\cK}\simeq \cX^{\cK}\hatbox {}_{\cK} \cT^{\cK}.
\]
In the above, ${}_{\cK} \cT^{\cK}$ is the transformer bimodule from \cite{ZemBordered}*{Section~14}. The above equation follows because the difference in the $h$-maps between $\cX^{\cK}$ and $\cX^{\cK}\hatbox{}_{\cK} \cT^{\cK}$ may also be factored through the map $\scU \Phi+\scV \Psi$. See the proof of \cite{ZemBordered}*{Theorem~14.1}.
\end{rem}

\section{Completing the proof of Theorem~\ref{thm:main}}
\label{sec:proof}

In this section, we complete our proof of Theorem~\ref{thm:main}. The organization is as follows. In Section~\ref{sec:HPL} we recall a version of the homological perturbation lemma for hypercubes. In Section~\ref{sec:relating-complexes}, we introduce a chain complex $\tilde{\cC}_{\Lambda}(G)$ and prove that there is a homotopy equivalence $\tilde{\cC}_{\Lambda}(G)\simeq \mathbb{CF}(G)$. The complex $\tilde{\cC}_{\Lambda}(G)$ is structurally similar to the link surgery complex for $L_G$. In Sections~\ref{sec:proof-b1=0} and ~\ref{sec:proof-b_1>0} we prove that $\tilde{\cC}_{\Lambda}(G)\simeq \cC_{\Lambda}(L_G)$, completing the proof of the equivalence between lattice homology and Heegaard Floer homology. Although the argument for $b_1(Y(G))>0$ also works when $b_1(Y(G))=0$, we have divided the proof into these two cases because the proof for $b_1(Y(G))=0$ is somewhat more natural and is useful in other situations (for example, it is necessary for the main results of \cite{BLZLatticeLink}).

Throughout this section we will assume that $G$ is connected. There is no loss of generality in doing so, since if $G=G_1\sqcup G_2$ then $Y(G)\iso Y(G_1)\# Y(G_2)$, and lattice homology and Heegaard Floer homology are tensorial under connected sum of 3-manifolds.

\subsection{Homological perturbation lemma for hypercubes}
\label{sec:HPL}
In this section, we review a version of the homological perturbation lemma for hypercubes. This lemma is similar to work of Huebschmann-Kadeishvili \cite{HK_Homological_Perturbation}. See  \cite{Liu2Bridge}*{Section~5.6} for a similar though slightly less explicit result for transferring hypercube structure maps along homotopy equivalences.

\begin{lem}[\cite{HHSZDuals}*{Lemma~2.10}]\label{lem:homological-perturbation-cubes}
Suppose that $\cC=(C_\veps,D_{\veps,\veps'})$ is a hypercube of chain complexes, and $(Z_\veps,\delta_\veps)_{\veps\in \bE_n}$ is a collection of chain complexes. Furthermore, suppose there are maps 
\[
\pi_{\veps}\colon C_\veps\to Z_\veps,\qquad i_{\veps}\colon Z_\veps\to C_\veps, \qquad h_\veps\colon C_\veps\to C_\veps,
\]
satisfying
\[
\pi_{\veps}\circ i_{\veps}=\id, \quad i_{\veps}\circ \pi_{\veps}=\id+[D_{\veps,\veps}, h_\veps], \quad  h_\veps\circ h_\veps=0,\]
\[
  \pi_{\veps}\circ h_{\veps}=0\quad \text{and} \quad h_{\veps}\circ i_{\veps}=0
\]
and such that $\pi_{\veps}$ and $i_{\veps}$ are chain maps.
With the above data chosen, there are canonical hypercube structure maps $\delta_{\veps,\veps'}\colon Z_{\veps}\to Z_{\veps'}$ so that $\cZ=(Z_\veps,\delta_{\veps,\veps'})$ is a hypercube of chain complexes, and also there are morphisms of hypercubes
\[
\Pi\colon \cC\to \cZ, \quad I\colon  \cZ\to \cC\quad \text{and}\quad H\colon \cC\to \cC
\]
such that
\[
\Pi\circ I=\id\quad \text{and} \quad I\circ \Pi= \id+\d_{\Mor}(H)
\]
and such that $I$ and $\Pi$ are chain maps.
\end{lem}

The structure maps $\delta_{\veps,\veps'}$ and the morphisms $\Pi$ and $I$ have a concrete formula. We begin with $\delta_{\veps,\veps'}$. Suppose that $\veps<\veps'$ are points in $\bE_n$.   The hypercube structure maps $\delta_{\veps,\veps'}$ are given by the following formula:
\[
\delta_{\veps,\veps'}:=\sum_{\veps=\veps_1<\cdots<\veps_j=\veps'}  \pi_{\veps'}\circ D_{\veps_{j-1},\veps_j}\circ h_{\veps_{j-1}}\circ D_{\veps_{j-2},\veps_{j-1}}\circ \cdots \circ h_{\veps_2}\circ D_{\veps_1,\veps_2}\circ i_{\veps}.
\]
The component of the map $I$ sending coordinate $\veps$ to coordinate $\veps'$ is given via the formula
\[
I_{\veps,\veps'}:=\sum_{\veps=\veps_1<\cdots<\veps_j=\veps'}  h_{\veps_{j}}\circ D_{\veps_{j-1},\veps_{j}}\circ \cdots \circ h_{\veps_2}\circ D_{\veps_1,\veps_2}\circ i_{\veps}.
\]
Similarly, $\Pi$ is given by the formula
\[
\Pi_{\veps,\veps'}:=\sum_{\veps=\veps_1<\cdots<\veps_j=\veps'}  \pi_{\veps'}\circ D_{\veps_{j-1},\veps_j}\circ h_{\veps_{j-1}}\circ D_{\veps_{j-2},\veps_{j-1}}\circ \cdots\circ D_{\veps_1,\veps_2} \circ h_{\veps_1}.
\]

\subsection{Relating the lattice and link surgery complexes}
\label{sec:relating-complexes}

In this section, we introduce a complex $\tilde{\cC}_{\Lambda}(G)$ which is similar to the link surgery complex. We prove in Proposition~\ref{prop:tilde-C=lattice} that there is a homotopy equivalence
\[
\tilde{\cC}_{\Lambda}(G)\simeq \mathbb{CF}(G).
\]

We begin by considering forests of trees $G$ such that each component has a distinguished vertex, which we label as the root. We consider the following operations on such collections of rooted trees, from which any rooted tree may be obtained:
\begin{enumerate}[label= ($G$-\arabic*), ref=$G$-\arabic*]
\item\label{graph-1} Adding a valence 0 vertex (viewed as the root of its component).
\item\label{graph-2}  Joining two components together at their roots.
\item\label{graph-3} Adding a valence 1 vertex at the root of a component, and making the new vertex the new root.
\end{enumerate}


 For connected $G$, we form a chain complex $\tilde{\cC}_{\Lambda}(G)$ as follows. We pick a vertex $v_0\in V(G)$ which we label as the root vertex. We define the chain complex $\tilde{\cC}_{\Lambda}(G)$ by iteratively tensoring the bordered modules and bimodules as follows, in parallel with the above topological moves:
\begin{enumerate}[label= ($M$-\arabic*), ref=$M$-\arabic*]
\item\label{module-1}
We begin with a copy of the type-$D$ module for a 0-framed solid torus $\cD_{0}^{\cK}$ for each valence 1 vertex of $G$ (other than $v_0$, if it has valence 1).
\item\label{module-2} We use the merge module ${}_{\cK| \cK} M^{\cK}$ to tensor the type-$D$ modules for two rooted trees to form the type-$D$ module for the tree obtained by joining the two components together at their roots.
\item\label{module-3} We tensor with the type-$DA$ module  of the Hopf link ${}_{\cK} \bar{\cH}_{(w(v),0)}^{\cK}$ to add a valence 1 vertex at the root of a tree. Here $w(v)$ denotes the weight of the vertex $v$. Recall that the bar indicates that we are using the alpha-beta bordered or beta-alpha bordered model of the link surgery formula for the Hopf link, as described in Section~\ref{sec:Hopf-link}.
\item\label{module-4} We tensor the type-$AA$ bimodule ${}_{\cK} [\cD_{w(v_0)}]_{\bF[U]}$ for the final root $v_0$.
\end{enumerate}

Note that since the bimodules are Alexander modules (in the sense of \cite{ZemBordered}*{Section~6}), the final type-$A$ action of $\bF[U]$ extends to an action of $\bF\llsquare U\rrsquare $ on the completion of $\tilde{\cC}_{\Lambda}(G)$.

%
%

\begin{rem} \emph{A-priori} $\tilde{\cC}_{\Lambda}(G)$ is not a valid model for the link surgery complex of $L_G$, though it is structurally similar. This is because in forming it, we used the bimodules ${}_{\cK|\cK} M^{\cK}$ when taking connected sums instead of the more complicated ${}_{\cK|\cK} W_l^{\cK}$ bimodules.
\end{rem}

  We prove the following:

\begin{prop}
\label{prop:tilde-C=lattice}
 The chain complexes $\tilde{\cC}_{\Lambda}(G)$ and $\bCF(G)$ are homotopy equivalent over $\bF\llsquare U\rrsquare$.
\end{prop}

\begin{proof}
We observe firstly that the link Floer complex of the Hopf link has the following filtration:
\[
\cCFL(H)=\big(\begin{tikzcd} \Span_{\bF[\scU_1,\scV_1,\scU_2,\scV_2]}(\ve{b},\ve{c})\ar[r, "\d"]& \Span_{\bF[\scU_1,\scV_1,\scU_2,\scV_2]}(\ve{a},\ve{d})
\end{tikzcd}\big).
\]
Here, $\ve{a}$, $\ve{b}$, $\ve{c}$ and $\ve{d}$ are the generators from Equation~\eqref{eq:Hopf-def}

Next, observe that $\tilde{\cC}_{\Lambda}(G)$ has a filtration by $\bE_\ell$ where $\ell=|V(G)|$, similar to $\bCF(G)$ and $\cC_{\Lambda}(L_G)$. That is, $\tilde{\cC}_{\Lambda}(G)$ is an $n$-dimensional hypercube of chain complexes.  Write $\tilde{\cC}_{\Lambda}(G)=(\tilde{\cC}_{\veps}, \tilde{D}_{\veps,\veps'})_{\veps\in \bE_n}$. 

For each $\veps$, consider the subcomplex $\tilde{\cC}_{\veps,\ve{s}}$ of $\tilde{\cC}_{\veps}$ which lies in a single Alexander grading $\ve{s}\in \bH(L_G)$.  The complexes $\cC_{\veps}$ are obtained by localizing a tensor product of  $(\ell-1)$-copies of the Hopf link complex $\cCFL(H)$ at some of the $\scV_i$ variables, and completing with respect to the $U_i$ and taking the direct product over Alexander gradings. In particular, $\tilde{\cC}_{\veps,\ve{s}}$ has an $\ell$-step filtration induced by the 2-step filtration on the link Floer complex of the Hopf link.  Write $\tilde{c}_{\veps,\ve{s}}$ for the chain complex $\tilde{\cC}_{\veps,\ve{s}}$ before completing with respect to the $U_i$ variables. The complex $\tilde{c}_{\veps,\ve{s}}$ may be written as the following exact sequence:
\begin{equation}
\tilde{c}_{\veps,\ve{s}}=\begin{tikzcd} 0\ar[r]&\cF^\ell_{\veps,\ve{s}}\ar[r]& \cdots \ar[r]& \cF^1_{\veps,\ve{s}}.
\end{tikzcd}
\label{eq:long-exact-sequence}
\end{equation}
We refer to the superscript $i$ in $\cF_{\veps,\ve{s}}^i$ as the \emph{Hopf filtration} level. In Equation~\eqref{eq:long-exact-sequence}, $\cF^i_{\veps,\ve{s}}$ is generated over $\bF[U_1,\dots, U_{\ell}]$ by elementary tensors in $S_{\veps}^{-1} \cdot \cCFL(L_G)$ which contain $i-1$ factors which are $\ve{b}$ or $\ve{c}$, and $\ell-i$ factors which are  $\ve{a}$ or $\ve{d}$, and which have Alexander grading $\ve{s}$. Recall that $S_{\veps}$ denotes the multiplicatively closed set generated by $\scV_i$ where $\veps_i=1$.

Index the link components so that the root vertex $v_0$ corresponds to the variable $U=U_1$. The homology group $H_*(\tilde{c}_{\veps,\ve{s}})$ is isomorphic to $\bF[U_1]$ by \cite{OSPlumbed}*{Lemma~2.6} (see also \cite{OSSLattice}*{Lemma~4.2}). Furthermore, the generator may be taken to be an elementary tensor which has only factors of $\ve{a}$ and $\ve{d}$, as well as some of the $\scU_i$ and $\scV_i$ variables. In particular, the generator is supported in $\cF_{\veps,\ve{s}}^1$. It follows that the chain complex $\tilde{c}_{\veps,\ve{s}}$ is a free resolution of its homology over $\bF[U_1,\dots, U_\ell]$.

 Since $H_*(\tilde{c}_{\veps,\ve{s}})\iso \bF[U_1]$ is a projective module over $\bF[U_1]$, it is a basic exercise in homological algebra to show that the above exact sequence may be split over $\bF[U_1]$. That is, we may decompose each $\cF_{\veps,\ve{s}}^{i}$ over $\bF[U_1]$ into a direct sum of $\bF[U_1]$-modules
 \[
 \cF_{\veps,\ve{s}}^i\iso L_{\veps,\ve{s}}^{i}\oplus R_{\veps,\ve{s}}^{i}
 \]
 so that $\d$ maps each $R^{i+1}_{\veps,\ve{s}}$ isomorphically onto $L^{i}_{\veps,\ve{s}}$, for $1<i<\ell$, and $\d$ vanishes on $L^{i}_{\veps,\ve{s}}$. Further, $L^{\ell}_{\veps,\ve{s}}=0$ and $R_{\veps,\ve{s}}^{1}$ projects isomorphically into $H_*(c_{\veps,\ve{s}})$. We may then define maps $\pi_{\veps,\ve{s}}$, $h_{\veps,\ve{s}}^{i}$ and $i_{\veps,\ve{s}}$ as in the following diagram
 \[
 \begin{tikzcd}[column sep=.4cm]
  R^{\ell}_{\veps,\ve{s}}
  	\ar[r, "\d",swap]
  &
  L^{\ell-1}_{\veps,\ve{s}}\oplus  R^{\ell-1}_{\veps,\ve{s}}
  \ar[r, "\d",swap]
  \ar[l, "h^{\ell-1}_{\veps}",bend right,swap, pos=.4]
  &
  \cdots
  \ar[l,bend right, swap,"h^{\ell-2}_{\veps,\ve{s}}", pos=.65]
  \ar[r,swap, "\d"]
  &
  L^{2}_{\veps,\ve{s}}
  \oplus
    R^{2}_{\veps,\ve{s}}
 \ar[r, "\d",swap]
 \ar[l,bend right,swap, "h^2_{\veps,\ve{s}}", pos=.3]
 &
 L^{1}_{\veps,\ve{s}}
 \oplus
 R^{1}_{\veps,\ve{s}}
  \ar[l,bend right,swap, "h^1_{\veps,\ve{s}}"]
  \ar[r,swap, "\pi_{\veps,\ve{s}}"]
 &H_*(\tilde{c}_{\veps,\ve{s}})
 	\ar[l, bend right, swap, "i_{\veps,\ve{s}}", pos=.6]
 \end{tikzcd}
 \]
The map $\pi_{\veps,\ve{s}}$ is the canonical projection map, and $i_{\veps,\ve{s}}$ is a section of the map $\pi_{\veps,\ve{s}}$, which we assume is compatible with the splitting described earlier. Similarly, 
\[
h_{\veps,\ve{s}}^i \colon L_{\veps,\ve{s}}^{i}\to R_{\veps,\ve{s}}^{i+1}
\] is the inverse of $\d|_{L_{\veps,\ve{s}}^{i+1}}$. Note that $i_{\veps,\ve{s}}$ and $h_{\veps,\ve{s}}^i$ are not generally $\bF[U_1,\dots, U_\ell]$-equivariant, and instead only $\bF[U_1]$-equivariant.

We may define $\bF\llsquare U_1\rrsquare $-equivariant maps
\[
\pi_{\veps}\colon \tilde{\cC}_{\veps}\to H_*(\tilde{\cC}_{\veps}),\quad  i_{\veps}\colon H_*(\tilde{\cC}_{\veps})\to \tilde{\cC}_{\veps}\quad \text{and} \quad  h_{\veps}\colon \tilde{\cC}_{\veps}\to \tilde{\cC}_{\veps}
\]
which give a homotopy equivalence over $\bF\llsquare U_1\rrsquare $ between $\tilde{\cC}_{\veps}$ and $H_*(\tilde{\cC}_{\veps})$, by taking the direct product of the maps $i_{\veps,\ve{s}}$, $\pi_{\veps,\ve{s}}$ and $h_{\veps,\ve{s}}=\sum_{i=1}^\ell h^i_{\veps,\ve{s}}$, and then completing with respect to the variables $U_1,\dots, U_\ell$. To see that the maps $i_{\veps,\ve{s}}$, $\pi_{\veps,\ve{s}}$ and $h_{\veps,\ve{s}}$ induce well-defined maps after completing with respect to $U_1,\dots, U_\ell$, we argue as follows. The completion over $U_1,\dots,U_\ell$ may be viewed as the completion with respect to the $I$-adic topology on $\bF[U_1,\dots, U_\ell]$, where $I$ is the ideal $(U_1,\dots, U_\ell)$. See \cite{AtiyahMacdonald}*{Chapter~10}. Equivalently, since there are only finitely many generators over $\bF[U_1,\dots, U_\ell]$ and each $U_i$ has Maslov grading $-2$, we may describe the completion as having a fundamental system of open sets given by $M_i=\Span\{x: \gr(x)\le i\}$ ranging over $i<0$.  The maps $i_{\veps,\ve{s}}$ and $h_{\veps,\ve{s}}$ are clearly continuous with respect to this topology, since they are homogeneously graded, and hence induce maps on the completion.

Applying the homological perturbation lemma for hypercubes, Lemma~\ref{lem:homological-perturbation-cubes}, we may transport the hypercube maps of $(\tilde{\cC}_{\veps}, \tilde{D}_{\veps,\veps'})$ to a hypercube with underlying groups $H_*(\tilde{\cC}_{\veps})$. This coincides with the underlying group of the lattice complex by Proposition~\ref{prop:OSS-lattice}. To see that the resulting hypercube is the lattice complex, it is sufficient to show that there are no higher length hypercube maps when we apply the homological perturbation lemma. This is seen directly, as follows. The hypercube maps induced by homological perturbation take the following form:
\[
\sum_{\veps=\veps_1<\cdots<\veps_j=\veps'}  \pi_{\veps'}\circ D_{\veps_{j-1},\veps_j}\circ h_{\veps_{j-1}}\circ D_{\veps_{j-2},\veps_{j-1}}\circ \cdots \circ h_{\veps_2}\circ D_{\veps_1,\veps_2}\circ i_{\veps}.
\]  Note that $h_{\veps_i}$ strictly increases the Hopf filtration level in the sequence in Equation~\eqref{eq:long-exact-sequence}, and $\pi_{\veps}$ is non-vanishing on only the lowest Hopf filtration level. Additionally, using Proposition~\ref{prop:Hopf-link-large-model} and the tensor product formula in Theorem~\ref{thm:pairing-main}, we see that the length 1 maps of $\tilde{\cC}_{\Lambda}(G)$ preserve the Hopf filtration level, while the higher length maps strictly increase the Hopf filtration level. In particular, the only way for a summand from Lemma~\ref{lem:homological-perturbation-cubes} to contribute is for there to be no $h_{\veps}$ factors, and for the hypercube arrow to be length 1. Hence, the transported hypercube maps are exactly the maps induced on homology by the length one maps of $\tilde{\cC}_{\Lambda}(G)$. This is exactly lattice homology by Proposition~\ref{prop:OSS-lattice}. The proof is complete.
\end{proof}

\begin{rem} The above argument fails at the last step if we use the $\cH$-models instead of the $\bar \cH$ models. (Recall that $\cH$ and $\bar{\cH}$ were introduced in Proposition~\ref{prop:Hopf-link-large-model}). This is because the length 2 map of $\cH$ has a term which decreases the Hopf filtration level, so could potentially contribute to higher length terms in the hypercube structure maps obtained from the homological perturbation lemma.
\end{rem}

\subsection{Completion of the proof when $b_1(Y(G))=0$}
\label{sec:proof-b1=0}

We begin with a basic topological lemma. Suppose that $G$ is a connected plumbing tree. Let $v$ be a vertex of $G$, and suppose that $v$ has valence $n>1$. Let us split $v$ into $n$ valence $1$ vertices, dividing $G$ into $n$ components $G_1,\dots, G_n$. See Figure~\ref{fig:5}. We let $Y_i$ denote the Dehn surgery on all weighted vertices of $G_i$ other than the one coming from $v$, so that $Y(G)$ is Dehn surgery on the knot $K_1\# \cdots\# K_n\subset Y_1\#\cdots \# Y_n$.

\begin{lem}
\label{lem:rational-homology-plumbings} Let $G$ be a plumbing tree with a vertex $v$ of valence $n>1$, and let $Y_1,\dots, Y_n$ be as above.  If $b_1(Y(G))=0$, then  $b_1(Y_1)+\cdots+b_1(Y_n)\le 1$. In particular, at most one $K_i$ is non-trivial in $H_1(Y_i)/\Tors$.
\end{lem}
\begin{proof}
We have $b_1(Y_1\# \cdots \# Y_n)=b_1(Y_1)+\cdots+b_1(Y_n)$. On the other hand, $Y(G)$ is obtained by performing Dehn surgery $K_1\#\cdots \# K_n\subset Y_1\#\cdots\# Y_n$ and $b_1(Y(G))=0$. Dehn surgery on a knot can reduce $b_1$ by at most one, so the claim follows.
\end{proof}

\begin{figure}[h]
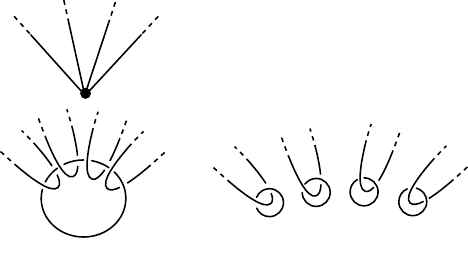
\caption{A vertex $v$ in a plumbing, and the knots $K_1$,\dots,$K_n$ from Lemma~\ref{lem:rational-homology-plumbings}. The manifold $Y_i$ is obtained by performing Dehn surgery on all vertices of $G_i$ except for the one shown.}
\label{fig:5}
\end{figure}

We now complete the the proof of the main theorem:

\begin{proof}[Proof of Theorem~\ref{thm:main} when $b_1(Y(G))=0$]
To construct $\cC_{\Lambda}(L_G)$, we  tensor bordered bimodules similar to ~\eqref{module-1}--\eqref{module-4}, as in our construction of $\tilde{\cC}_{\Lambda}(G)$. The only difference is that instead of the merge module $ M$, we must use the pair-of-pants modules $W_r$ or $W_l$. Note that we may construct $Y(G)$ by iteratively performing topological moves parallel to the algebraic modules~\eqref{module-1}--\eqref{module-4}. By Lemma~\ref{lem:rational-homology-plumbings}, when constructing plumbed 3-manifolds which are rational homology 3-spheres by iteratively taking connected sums and taking duals, we will never take the connected sum of two knots which are both homologically essential. Hence, we may apply Corollary~\ref{cor:simplify-tensor-product} to simplify the complexes of connected sums by choosing one of the two pair-of-pants modules ($W_l$ or $W_r$) so that the homology action term is applied to the tensor factor from the rationally null-homologous knot. This allows us to replace the pair-of-pants modules with the merge module without changing the homotopy type of the resulting type-$D$ module over $\cK$. In particular, we see that for some choice of arc system $\scA$ on $L_G$, we have
\[
\cC_{\Lambda}(L_G,\scA)_{\bF[U]}\simeq \tilde{\cC}_{\Lambda}(L_G)_{\bF[U]},
\]
so the main result follows from Proposition~\ref{prop:tilde-C=lattice}.

We now address the claim about the relative grading. We recall from \cite{MOIntegerSurgery}*{Section~9.3} that when $b_1(Y(G))=0$, the link surgery hypercube possesses a uniquely specified relative $\Z$-grading on each $\Spin^c$ structure. Furthermore, this relative grading clearly coincides with the relative grading on the lattice complex. Note that in our proof, we showed that the link surgery complex $\cC_{\Lambda}(L_G)$ coincided with the tensor product of the bordered modules in~\eqref{module-1}--\eqref{module-4} only up to homotopy equivalence. However in the case that $b_1(Y(G))=0$, the maps appearing in this homotopy equivalence only involve maps which already appear in the differential of the link surgery formula, or involve projections onto complexes in different $\veps$ and $\ve{s}\in \bH(L)$. In particular, the homotopy equivalence will be grading preserving.
\end{proof}

\subsection{Completion of the proof when $b_1(Y(G))>0$}

\label{sec:proof-b_1>0}

\begin{proof}[Proof of Theorem~\ref{thm:main} when $b_1(Y(G))>0$]
The proof is much the same as when $b_1(Y(G))=0$. The main difference is relating tensor products with the pair-of-pants bimodules with tensor products using the merge module. We can no longer use Corollary~\ref{cor:simplify-tensor-product} to eliminate the terms involving the homology action. This is because when $b_1(Y(G))>0$ we may, at an intermediate step, need to take the connected sum of two homologically essential knots (cf.  Lemma~\ref{lem:rational-homology-plumbings}). If at most one of the knots is homologically essential, then we can pick judiciously which model of $W_r$ or $W_l$ to use, so that the homology action term $\cA_{K_i}$ which appears is on the knot which is homologically trivial. If more than one of these knots is homologically essential, then this technique will not apply.

In the present case, we apply Propositions~\ref{prop:merge=pair-of-pants} and ~\ref{prop:finite-generation} instead of  Corollary~\ref{cor:simplify-tensor-product} when we need to take the connected sum of two components, as in \eqref{module-2}.  Propositions~\ref{prop:finite-generation} shows that relevant type-$D$ modules are finitely generated (up to homotopy equivalence), and Proposition~\ref{prop:merge=pair-of-pants} allows us to replace the modules ${}_{\cK|\cK}W_r^{\cK}$ in the tensor product with ${}_{\cK|\cK} M^{\cK}$. From here the proof is identical to the case when $b_1(Y(G))=0$. 
\end{proof}

\begin{rem} When $b_1(Y(G))>0$, the homotopy equivalence relating the tensor products obtained using the merge modules and the pair-of-pants bimodules is obtained using Corollary~\ref{cor:eliminate-homology-action}, which is fundamentally based on the classification theorem for finitely generated chain complexes over a PID. In particular, it is not as clear from this perspective how the homotopy equivalence interacts with the relative grading on the link surgery formula. We will pursue this question in a future work by considering group valued gradings on the bordered link surgery modules, in the spirit of \cite{LOTBordered}*{Section~3.3}.
\end{rem}

\section{Plumbed 3-manifolds with $b_1>0$}
\label{sec:plumbed-H1-action}
In this section, we state a refinement of N\'{e}methi's conjecture for plumbed manifolds with $b_1>0$. In analogy to the action of $H_1(S^3_\Lambda(L))/\Tors$ on the link surgery complex from Section~\ref{sec:algebraic-action-link-surgery}, we now describe an action of $H_1(Y(G))/\Tors$ on lattice homology. If $\mu_i$ denotes the meridian of the component $K_i\subset L_G$, we define
\[
\frA_{\mu_i}([K,E])=U^{a_{v_i}[K,E]} [K,E-v_i],
\]
extended equivariantly over $U$. (This is a term in the differential $\d$ on the lattice complex).
If $\g=a_1\cdot \mu_1+\cdots+ a_\ell \cdot\mu_\ell\in H_1(Y(G))/\Tors$, we define
\[
\frA_{\g}:=a_1\frA_{\mu_1}+\cdots+ a_\ell \frA_{\mu_\ell}.
\]
\begin{lem} The above endomorphism $\frA_{\g}$ of $\bCF(G)$ depends only on the class $\g\in H_1(Y(G))/\Tors$, up to chain  homotopy.
\end{lem}
The proof is essentially identical to the proof of Lemma~\ref{lem:algebraic-homology-action-link-surgery} for the link surgery formula. One replaces the lattice $\bH(L)$ with $\Char(X(G))$ and the proof proceeds by an easy translation.

\begin{conj} $\ve{\HF}^-(Y(G))$ is isomorphic to $\mathbb{HF}(G)$ as a module over $\bF\llsquare U\rrsquare \otimes \Lambda^* H_1(Y(G))/\Tors$.
\end{conj}

\bibliographystyle{custom}
\def\MR#1{}
\bibliography{biblio}

\end{document}